\documentclass[11pt,oneside,reqno]{amsart}

\usepackage{amsmath}
\usepackage{graphicx}
\usepackage{amssymb}
\usepackage{amsbsy}
\usepackage{bm}

\usepackage{nicefrac}
\usepackage{float}
\usepackage{epsfig}
\input epsf
\newtheorem{theorem}{Theorem}[section]
\newtheorem{corollary}[theorem]{Corollary}

\newtheorem{lemma}[theorem]{Lemma}

\newtheorem{remark}{Remark}[section]

\def\hdiv{H(\mbox{div\,},{\mathcal D})}
\def\hdivN{H_{\Gamma_N}(\mbox{div\,},{\mathcal D})}
\def\L2{L^2({\mathcal D})}
\def\h01{H^1_0({\mathcal D})}
\def\h10{H^1_0({\mathcal D})}
\def\tO{\widetilde{\mathcal D}}
\def\tphi{\widetilde\phi}

\def\b0{{\pmb 0}}
\def\bn{{\pmb n}}
\def\bs{{\pmb\sigma}}
\def\bst{\tilde{\pmb\sigma}}
\def\bt{{\pmb\tau}}
\def\Cp{{\mathcal P}}

\def\RT{{\mathcal {RT}}}
\def\cth{{\mathcal T}_h}

\begin{document}
\date{}
\title[Mixed methods for degenerate elliptic problems]
{Mixed methods for degenerate elliptic problems
and application to fractional laplacian}

\author{Mar\'\i a E. Cejas}
\address{Departamento de Matem\'atica \\
    Facultad de Ciencias Exactas\\
    Universidad Nacional de La Plata\\
    Calle 50 y 115\\
(1900) La Plata, Prov. de Buenos Aires  \\
    Argentina.}
\email{mec.eugenia@gmail.com}

\author{Ricardo G. Dur\'an}
\address{IMAS (UBA-CONICET) and Departamento de Matem\'atica\\
    Facultad de Ciencias Exactas y Naturales\\
    Universidad de Buenos Aires\\
    Ciudad Universitaria\\
    (1428) Ciudad Aut\'onoma de Buenos Aires\\
    Argentina.}
\email{rduran@dm.uba.ar}

\author{Mariana I. Prieto}
\address{INMABB (UNS-CONICET) and Departamento de Matem\'atica\\
    Universidad Nacional del Sur\\
    Av. Alem 1253\\
    (8000) Bah\'\i a Blanca, Prov. de Buenos Aires\\
    Argentina.}
\email{miprieto@uns.edu.ar}

\thanks{Supported by ANPCyT under grant PICT 2014-1771, by CONICET under grant 11220130100006CO  and by Universidad de Buenos Aires under grant 20020120100050BA. The first author has a fellowship from CONICET, Argentina.}



\begin{abstract}
    We analyze the approximation by mixed finite element methods of solutions of
    equations of the form $-\mbox{div\,} (a\nabla u) = g$, where the coefficient $a=a(x)$ can
    degenerate going to cero or infinity.
    First, we extend the classic error analysis to this case provided that the
    coefficient $a$ belongs to the Muckenhoupt class $A_2$.
    The analysis developed applies to general mixed finite element spaces satisfying the
    standard commutative diagram property, whenever some stability and interpolation
    error estimates are valid in weighted norms. Next, we consider in detail the case
    of Raviart-Thomas spaces of lowest order, obtaining optimal order error estimates for
    general regular elements as well as for some particular anisotropic ones which are of
    interest in problems with boundary layers. Finally we apply the results to a problem
    arising in the solution of the fractional Laplace equation.
\end{abstract}

\maketitle

\section{Introduction}
\label{intro}
\setcounter{equation}{0}
In this paper we analyze the approximation by mixed finite element
methods of degenerate second order elliptic problems. There is a
vast bibliography concerning this kind of methods (see for example
the books \cite{BBF, BBDDFF} and references therein). However, as
far as we know, only very few papers have considered the
degenerate case (we can mention \cite{A,MP}).

Let ${\mathcal D}\subset{\mathbb R}^n$ be a bounded Lipschitz polytope
and $a$ be a non-negative measurable function.
We assume that the boundary is decomposed into two disjoint parts
$\Gamma_D$ and $\Gamma_N$.
Given $g\in L^2({\mathcal D})$ and $f\in L^2(\Gamma_N)$ we consider the problem

\begin{equation}
\label{ep}
\left\{
\begin{array}{rl}
-\mbox{div\,} (a\nabla u) = g &\ \ \mbox{in} \ {\mathcal D}\\
u=0 & \ \ \mbox{on} \ \Gamma_D\\
-a\nabla u\cdot\bn=f &\ \ \mbox{on}\ \Gamma_N
\end{array}
\right.
\end{equation}
where $\bn$ denotes the unit exterior normal vector.
If $\Gamma_N=\partial{\mathcal D}$ we assume the usual compatibility
condition $\int_{\mathcal D} g=\int_{\partial{\mathcal D}}f$.

We have written the problem in this form in order to simplify notation.
However, it is easy to see that
all our arguments apply to general problems where
the coefficient $a$ is replaced by a matrix $A=A(x)$
satisfying
$\lambda a(x)|\xi|^2\le \xi^TA(x)\xi\le\Lambda a(x) |\xi|^2$,
for all $x\in{\mathcal D}$,
where  $\lambda$ and $\Lambda$ are positive constants.

We are interested in degenerate problems in the sense that the
coefficient $a$ can become infinite or zero in subsets of
$\overline{\mathcal D}$ with vanishing $n-$dimensional measure.
We will assume that $a$ belongs to the Muckenhoupt class $A_2$,
in particular $a^{-1}\in L^1_{loc}({\mathcal D})$ and, therefore,
the usual mixed method is well defined.

Recall that a non-negative measurable function
$a\in L^1_{loc}({\mathbb R}^n)$ belongs to $A_2$ if
$$
[a]_{A_2}:=\sup_{Q} \left(\frac{1}{|Q|}
\int_Q a \right)\left(\frac{1}{|Q|}
\int_ Q a^{-1} \right)<\infty,
$$
where the supremum is taken over all cube $Q$ with
faces parallel to the coordinate axes.

The class $A_2$ was introduced to characterize
the weights for which the Hardy-Littlewood maximal operator is bounded in the associated weighted norm
(See for instance \cite{CF,M}). After that, it was used in the theory of elliptic equations
(see for example the pioneering work \cite{FKS}) and, more recently,
in the analysis of finite element approximations \cite{AGM,NOS,NOS2}.

When dealing with anisotropic estimates we will work with the
more restrictive {\it strong $A_2$} class, which will be denoted by
$A^s_2$ and is defined by
$$
[a]_{A^s_2}:=\sup_{R} \left(\frac{1}{|R|}
\int_R a \right)\left(\frac{1}{|R|}
\int_R a^{-1}\right)<\infty.
$$
where the supremum is taken now over all $n$-dimensional rectangles
with faces parallel to the coordinate axes. It is known that
$a\in A^s_2$ if and only if $a$ belongs to $A_2$ of one variable
for each variable, uniformly in the other variables (see \cite{GCRF, K}).

Given a weight $a$, for any measurable set $S$ we will denote with $L_a^2(S)$ the
usual Hilbert space with measure $a\,dx$. We will also work
with the weighted Sobolev space
$$
H_a^1(S)=\left\{v\in L_a^2(S)\,:\, |\nabla v|\in L_a^2(S)\right\}
$$
with its natural norm. We will omit the domain in these notations
when it is clear from the context.

Under appropriate assumptions on $a$ (particularly if $a\in A_2$)
and the data $f$ and $g$, it is possible to prove by standard arguments that there
exists a unique solution of problem
\eqref{ep} belonging to $H_a^1({\mathcal D})$.

Introducing the variable vector field $\bs=-a\nabla u$,
problem (\ref{ep}) can be transformed into the equivalent
first order system
\begin{equation}
\label{fos}
\left\{
\begin{array}{rl}
\bs + a \nabla u = 0 &\ \ \mbox{in} \ {\mathcal D}\\
\mbox{div\,}\bs=g&\ \ \mbox{in} \ {\mathcal D}\\
u=0 & \ \ \mbox{on} \ \Gamma_D\\
\bs\cdot\bn=f &\ \ \mbox{on}\ \Gamma_N
\end{array}
\right.
\end{equation}
Then, mixed finite element methods are based on a weak formulation
of this system and they approximate simultaneously $\bs$ and $u$.
One motivation for using this type of methods is that, in many applications,
the variable of physical interest is $\bs$ and, therefore, it might be
more efficient to approximate it directly instead of obtaining it from
a computed approximation of $u$. A typical example of this situation
is the Darcy equation arising in the simulation of flows in porous media.
Indeed, it is many times argued that $\bs$ is smoother than
$\nabla u$. Although this is probably true in practice, it is not possible to give a
mathematical foundation to this statement in general (see \cite{FO} for an interesting
discussion on this subject).

As an application of our results we will consider
a problem arising in the solution of the fractional Laplace equation $(-\Delta)^sv=f$.
As we will show, in the case $0<s<1/2$, the mixed method is
more convenient than the standard one in the sense that almost optimal order of
convergence can be obtained with a weaker grading of the meshes.

The rest of the paper is organized as follows.
In Section \ref{mixed approximations} we recall the mixed finite element method
for Problem \eqref{ep} and extend the classic error analysis to the case of degenerate
problems. A fundamental tool is the existence of right inverses of the divergence
in weighted norms when the weight belongs to the class $A_2$. The analysis given in
this section can be applied to general mixed finite element spaces which satisfy
the so called commutative diagram property whenever a
stability property in a weighted norm for the interpolation operator is valid.
Next, in Section \ref{RT}, we consider the case of Raviart-Thomas elements of lowest
order and prove the stability property mentioned above and error estimates
in weighted norms under the regularity assumption on the family of meshes.
Then, in Section \ref{RT anisotropicos}, we continue the analysis for
the Raviart-Thomas spaces of lowest order and prove some weighted interpolation error
estimates, where the weights involve the distance to some part of the boundary, for anisotropic rectangular and prismatic
elements which are of interest
in problems with boundary layers. An important tool in this part of the analysis is the so called
improved Poincar\'e inequality. Finally, in Section \ref{fraccionario},
we consider the approximation of the fractional Laplace equation which leads to
a particular degenerate problem of the type considered in the previous sections.
We show in this example how the weighted error estimates proved for anisotropic elements can
be used to design a priori adapted meshes giving almost optimal order with respect to
the number of degrees of freedom. We include in this section some numerical results.

\section{Mixed finite element approximations}
\label{mixed approximations}
\setcounter{equation}{0}
First we recall
some usual notation and known results
on mixed methods. The appropriate space for the vector
variable is
$$
\hdiv = \{\bt\in L^2({\mathcal D})^n : \mbox{div\,}\bt \in L^2({\mathcal D})\}
$$
which is a Hilbert space with norm given by
$$
\|\bt\|^2_{\hdiv}=\|\bt\|^2_{\L2} + \|\mbox{div\,}\bt\|^2_{\L2}.
$$
Moreover, since in the mixed formulation Neumann type boundary
conditions are imposed in an essential way, we will work with
the subspace
$$
\hdivN =\{\bt\in\hdiv  : \, \bt\cdot\bn=0 \ \ \mbox{on}\ \ \Gamma_N\}.
$$
Dividing by $a$, the first equation in (\ref{fos}) can
be rewritten as
$$
a^{-1}\,\bs + \nabla u = 0 \ \ \mbox{in} \ {\mathcal D},
$$
and multiplying by test functions and integrating by parts, we obtain
the standard weak mixed formulation of problem (\ref{fos}),
namely, find $\bs\in\hdiv$ and $u\in\L2$ such that
\begin{equation}
\label{Neumann}
\bs\cdot\bn=f \ \quad \mbox{on}\ \ \Gamma_N
\end{equation}
and
\begin{equation}
\label{mf}
\left\{
\begin{aligned}
\int_{{\mathcal D}} a^{-1}\,\bs\cdot\bt - \int_{{\mathcal D}} u\, \mbox{div\,}\bt
&= 0
&\forall \bt \in \hdivN\\
\int_{{\mathcal D}} v\,\mbox{div\,}\bs
&= \int_{{\mathcal D}} g v
&\forall v\in\L2
\end{aligned}
\right.
\end{equation}
Observe that the Dirichlet boundary condition is implicit in the weak
formulation. When $\Gamma_N=\partial{\mathcal D}$, $\L2$ has to be replaced by $L^2_0({\mathcal D})$,
the subspace of functions with vanishing mean value.

As usual, the error analysis is divided in two steps. The first one consists
in proving estimates for the finite element approximation error in terms
of the error for some appropriate interpolation or projection operator.
This part of the analysis can be done for general mixed finite
element spaces provided they satisfy the so called commutative diagram
property as well as some weighted stability estimates for the appropriate projections.
Therefore, we will develop this part of the error analysis for general spaces stating
the necessary assumptions that afterwards have to be proved for each particular choice
of approximation spaces. The second part consists in estimating the interpolation error.
For simplicity, we will restrict this analysis to the lowest order Raviart-Thomas elements. Higher order
elements as well as other approximation spaces could be treated similarly but this
require non trivial technical modifications.

We assume that we have a family of partitions $\{\cth\}$ of the domain ${\mathcal D}$
such that each $\cth$
is consistent with the boundary conditions, i. e.,
the exterior boundary of an element is completely contained in $\Gamma_D$
or in $\Gamma_N$.
Associated with these partitions we assume that we have finite element spaces
${\pmb S}_h\subset\hdiv$, $V_h\subset\L2$ (or $V_h\subset L^2_0({\mathcal D})$ when $\Gamma_N=\partial{\mathcal D}$),
such that, if
$$
{\pmb S}_{h,N}={\pmb S}_h\cap\hdivN,
$$
then
\begin{equation}
\label{div S=V}
\mbox{div\,}{\pmb S}_{h,N}=V_h
\end{equation}
and there exists an operator $\Pi_h:{\pmb S}\longrightarrow{\pmb S}_h$, defined in an appropriate subspace
${\pmb S}\subset\hdiv$ containing the solution $\bs$, such that, if $\bt\in{\pmb S}\cap\hdivN$ then
$\Pi_h\bt\in{\pmb S}_{h,N}$ and, for all $\bt\in{\pmb S}$,
\begin{equation}
\label{prop fundamental}
\int_{\mathcal D}\mbox{div\,}(\bt-\Pi_h\bt)v=0
\qquad \forall v\in V_h.
\end{equation}
Introducing the $L^2$-orthogonal projection
$P_h: L^2({\mathcal D})\longrightarrow V_h$,
\eqref{div S=V} and \eqref{prop fundamental} yield the commutative diagram property
\begin{equation}
\label{propiedad conmutativa}
\mbox{div\,}\Pi_h=P_h\mbox{div\,}.
\end{equation}
The mixed finite element approximation of problem (\ref{fos})
is given by
$$
(\bs_h, u_h)\in {\pmb S}_h\times V_h
$$
such that,
\begin{equation}
\label{discrete Neumann}
\bs_h\cdot\bn=\Pi_h\bs\cdot\bn \quad\mbox{on}\ \ \Gamma_N
\end{equation}
and
\begin{equation}
\label{mixedfem}
\left\{
\begin{aligned}
\int_{{\mathcal D}} a^{-1}\,\bs_h\cdot\bt - \int_{{\mathcal D}} u_h\, \mbox{div\,}\bt
&= 0
&\forall \bt \in {\pmb S}_{h,N}, \\
\int_{{\mathcal D}} v\,\mbox{div\,}\bs_h,
&= \int_{{\mathcal D}} g v
&\forall v\in V_h.
\end{aligned}
\right.
\end{equation}
Existence and uniqueness of the discrete solution and the following
error estimate follow by well known arguments (see for example \cite{BBF,BBDDFF}).
For completeness we include the proof of the error estimate to show that
the usual arguments can be adapted for degenerate problems and for the mixed boundary conditions considered here.
We neglect numerical integration errors assuming that all the integrals
can be computed exactly.

\begin{lemma}
\label{error sigma}
Assume that $a^{-1}\in L^1({\mathcal D})$ and $\bs\in L^2_{a^{-1}}({\mathcal D})^n$.
If $\bs$ is the solution of (\ref{Neumann}) and (\ref{mf})and $\bs_h$
that of  (\ref{discrete Neumann}) and (\ref{mixedfem}),
then
$$
\|\bs - \bs_h\|_{L^2_{a^{-1}}}
\le \|\bs-\Pi_h\bs\|_{L^2_{a^{-1}}}.
$$
\end{lemma}
\begin{proof} Subtracting the second equation
in (\ref{mixedfem}) to the second one in (\ref{mf}) and using
(\ref{prop fundamental}) we obtain
$$
\int_{{\mathcal D}}\mbox{div\,}(\Pi_h\bs - \bs_h)\, v
=0  \ \ \forall v\in V_h .
$$
From (\ref{discrete Neumann})
it follows that $\Pi_h\bs - \bs_h\in{\pmb S}_{h,N}$, and then, by \eqref{div S=V}
we conclude that $\mbox{div\,}(\Pi_h\bs - \bs_h)=0$.
Moreover, taking ${\bt}=\Pi_h{\bs} - {\bs}_h$ in (\ref{mf}) and (\ref{mixedfem}),
we obtain
$$
\int_{{\mathcal D}} a^{-1}\, (\bs-\bs_h) \cdot(\Pi_h \bs-\bs_h)=0
$$
and so,
$$
\begin{aligned}
\|\bs-{\bs}_h\|_{L^2_{{a^{-1}}}}^2
&= \int_{{\mathcal D}} a^{-1}\, (\bs-\bs_h)\cdot
(\bs-\Pi_h{\bs}) \\
&\le \|\bs-\bs_h\|_{L^2_{a^{-1}}} \|\bs-\Pi_h\bs\|_{L^2_{a^{-1}}},
\end{aligned}
$$
and the lemma is proved.\end{proof}

To estimate the error in the approximation of the scalar variable
we need a stronger assumption on the coefficient $a$. Indeed, we will prove
the following result that generalizes to the weighted case
the existence of continuous right inverses of the divergence.

\begin{lemma}
\label{inversa de la div}
If $a\in A_2$ then, given $\phi\in L^2_{a^{-1}}({\mathcal D})$
(satisfying $\int_{\mathcal D}\phi=0$ in the case $\Gamma_N=\partial{\mathcal D}$),
there exists $\bt\in H_{a^{-1}}^1({\mathcal D})^n\cap\hdivN$ such that
$$
\mbox{div\,}\bt=\phi
$$
and
$$
\|\bt\|_{H_{a^{-1}}^1({\mathcal D})}\le C \|\phi\|_{L_{a^{-1}}^2({\mathcal D})},
$$
where the constant $C$ depends on ${\mathcal D}$ and $a$.
\end{lemma}
\begin{proof}
In the case $\Gamma_N=\partial{\mathcal D}$ we have $\int_{\mathcal D}\phi=0$ and the
result is known. Indeed, for domains which are star-shaped with
respect to a ball it was proved
in \cite[Th. 3.1]{DL} and \cite[Th.1.1]{Sch} using Bogovskii's solution of the divergence
and the theory of singular integrals. The arguments used there can be extended
for the class of John domains using
the generalization of Bogovskii's operator introduced in \cite{ADM} (For more details see also \cite{AD2}).
A different proof was given in \cite[Th. 5.2]{DRS} also for the class
of John domains.

Suppose now that $\Gamma_N\neq\partial{\mathcal D}$. Enlarging the domain in
an appropriate way we can obtain a Lipschitz domain $\tO$
such that ${\mathcal D}\varsubsetneq\widetilde{\mathcal D}$ and
$\Gamma_N\subset\partial\tO$. For example, we can make a smooth deformation
of part of $\Gamma_D$.

Now, we extend $\phi$ to $\tO$ as
$$
\tphi=
\left\{
\begin{aligned}
\phi(x) &\ \ ,\ \ x\in{\mathcal D}\\
-\frac{\int_{\mathcal D}\phi}{|\tO\setminus{\mathcal D}|} &\ \ ,\ \ x\in\tO\setminus{\mathcal D}
\end{aligned}
\right.
$$
and then, since $\int_{\tO}\tphi\,dx=0$, there exists $\bt\in H_{a^{-1}}^1(\tO)^n$,
vanishing on $\partial\tO$
and satisfying
$$
\|\bt\|_{H_{a^{-1}}^1(\tO)}\le C \|\tphi\|_{L_{a^{-1}}^2(\tO)}
$$
It is easy to see that $\|\tphi\|_{L_{a^{-1}}^2(\tO)}\le C \|\phi\|_{L_{a^{-1}}^2({\mathcal D})}$,
and, therefore, the restriction of $\bt$ to ${\mathcal D}$ satisfies the required properties.
\end{proof}

For the next lemma we need to use the following stability result in a weighted norm:
\begin{equation}
\label{prop4}
\|\Pi_h\bt\|_{L^2_{a^{-1}}}\le C \|\bt\|_{H_{a^{-1}}^1}.
\end{equation}
Assuming that
$a\in A_2$, we will
prove this estimate for the lowest order Raviart-Thomas spaces
in a forthcoming section.

\begin{lemma}
\label{error u}
Let $(\bs,u)$ and $(\bs_h,u_h)$ be the solutions of (\ref{Neumann}) and (\ref{mf}),
and (\ref{discrete Neumann}) and (\ref{mixedfem}) respectively. If $a\in A_2$
and $\Pi_h$ satisfies \eqref{prop4} then
\begin{equation}
\label{error u 2}
\|u-u_h\|_{L_a^2}
\le \|u-P_hu\|_{L_a^2}
+C\|\bs-\bs_h\|_{L_{a^{-1}}^2},
\end{equation}
where $C$ depends on the constant in Lemma \ref{inversa de la div}.
\label{lemacotap}
\end{lemma}

\begin{proof} Assume first that $\Gamma_D\neq\emptyset$.
According to Lemma \ref{inversa de la div}
there exists $\bt\in H_{a^{-1}}^1({\mathcal D})^n\cap\hdivN$ such that
$$
\mbox{div\,}\bt=(P_hu-u_h)a
$$
and
$$
\|\bt\|_{H_{a^{-1}}^1}\le C \|(P_hu-u_h)a\|_{L_{a^{-1}}^2}.
$$
Then,
$$
\begin{aligned}
\|P_hu-u_h\|^2_{L_a^2}
&=\int_{\mathcal D} (P_hu-u_h) \mbox{div\,}\bt
=\int_{\mathcal D} (P_hu-u_h) \mbox{div\,}\Pi_h\bt\\
&=\int_{\mathcal D} (u-u_h) \mbox{div\,}\Pi_h\bt
=\int_{\mathcal D}{a^{-1}} (\bs-\bs_h) \cdot\Pi_h\bt
\\
&\le \|\bs-\bs_h\|_{L^2_{a^{-1}}} \|\Pi_h\bt\|_{L^2_{a^{-1}}}\\
&\le C\|\bs-\bs_h\|_{L^2_{a^{-1}}} \|\bt\|_{H^1_{a^{-1}}}
\\
&\le C\|\bs-\bs_h\|_{L^2_{a^{-1}}}
\|P_hu-u_h\|_{L_a^2}
\end{aligned}
$$
where we have used \eqref{prop fundamental} and (\ref{prop4}). Then,
\eqref{error u 2} follows by the triangular inequality.

Now, if $\Gamma_N=\partial{\mathcal D}$, there exists $\bt\in\hdivN$ such that
$$
\mbox{div\,}\bt=(P_hu-u_h)a -  \overline{(P_hu-u_h)a},
$$
where $\overline{(P_hu-u_h)a}$ denotes the average of
$(P_hu-u_h)a$, and
$$
\|\bt\|_{H_{a^{-1}}^1}\le C \|(P_hu-u_h)a\|_{L_{a^{-1}}^2}.
$$
Indeed, this follows from Lemma \ref{inversa de la div} and the estimate
$$
|\overline{(P_hu-u_h)a}|
\le \frac1{|{\mathcal D}|}\left(\int_{\mathcal D}a\right)^{1/2}
\|(P_hu-u_h)a\|_{L_{a^{-1}}^2}.
$$
Since $\int_{\mathcal D} (P_hu-u_h)=0$ we have
$$
\begin{aligned}
\|P_hu-u_h\|^2_{L_a^2({\mathcal D})}
&=\int_{\mathcal D} (P_hu-u_h)\left((P_hu-u_h)a -  \overline{(P_hu-u_h)a}\right)\\
&=\int_{\mathcal D} (P_hu-u_h) \mbox{div\,}\bt.
\end{aligned}
$$
The rest of the argument follows as in the previous case.
\end{proof}
Combining Lemmas \ref{error sigma} and \ref{error u} we obtain the following
\begin{corollary}
Under the same hypotheses of Lemma \ref{error u} we have
$$
\|u- u_h\|_{L^2_{a}} \le  \|u-P_hu\|_{L_a^2}
+ C\|\bs-\Pi_h\bs\|_{L_{a^{-1}}^2}.
$$
\end{corollary}

\section{Error estimates for Raviart-Thomas elements}
\label{RT}
\setcounter{equation}{0}

We now consider the approximation by the lowest order Raviart-Thomas
mixed finite elements. To apply the results obtained in the previous section
we have to prove error estimates for the corresponding operators
$\Pi_h$ and $P_h$.

Recall that the local Raviart-Thomas space of lowest degree
for a simplex $T$ is
$$
\RT_0(T)=\{\bt\,:\,
\bt(x)=(a_1+bx_1,\cdots,a_n+bx_n)\ \ \mbox{with} \ \ a_i,b\in{\mathbb R}\}
$$
while for $R$ an $n$-dimensional rectangular element with faces parallel to
the coordinate axes, is
$$
\RT_0(R)=\{\bt\,:\,
\bt(x)=(a_1+b_1x_1,\cdots,a_n+b_nx_n)\ \ \mbox{with} \ \ a_i,b_i\in{\mathbb R}\}.
$$
Then, the global space for the mixed approximation of the vector variable
for a partition $\cth$ made of any kind of elements is
\begin{equation}
\label{Raviart-Thomas}
{\pmb S}_h=\{\bt\in \hdiv \colon \bt|_K \in \RT_0(K)
\ \ \forall K\in\cth\}.
\end{equation}
The associated space for the scalar variable is given by the piecewise constant functions, namely,
\begin{equation}
\label{Constantes a trozos}
V_h=\{v\in V \colon v|_K \in \Cp_0(K)
\ \ \forall K\in\cth\},
\end{equation}
where $V=L^2_0({\mathcal D})$ when $\Gamma_N=\partial{\mathcal D}$ or
$V=\L2$ otherwise. Then, the projection $P_h$ is given by $(P_hv)|_K=P_Kv:=\frac1{|K|}\int_Kv$.
The fundamental tool for the error analysis is the well known
Raviart-Thomas operator  $\Pi_h$ defined on each element $K$ as $(\Pi_h\bt)|_K=\Pi_K\bt$ where
\begin{equation}
\label{def de pi}
\int_F \Pi_K\bt\cdot\bn_F =\int_F \bt\cdot\bn_F
\end{equation}
for all face $F$ of $K$ where $\bn_F$ denotes a unitary vector normal to $F$ (here $K$ may be a simplex
or a rectangle). This operator is well
defined whenever the $\bt\cdot\bn\in L^1(F)$ which is known to be true for any
$\bt\in W^{1,1}(K)^n$. Moreover, it is not difficult to check that \eqref{div S=V},
\eqref{prop fundamental}, and consequently \eqref{propiedad conmutativa}, are satisfied.

We consider first the case of regular partitions, namely, if $h_K$ and $\rho_K$ are the diameters
of $K$ and the biggest ball contained in $K$ respectively, we assume that the family of meshes
$\{\cth\}$ satisfy $h_K/\rho_K\le\eta$ with a constant $\eta$ independent of $h$.

Basic tools for interpolation error estimates are the
Poincar\'e type inequalities. Given a set $S$ and a function $v$ we will denote with $v_S$
the average of $v$ over $S$ (both for $S=K$ or some face of $K$). In what follows the
constant $C$ will depend on the weight $[a]_{A_2}$, although it is possible to give an
explicit bound for this dependence this is not of interest for our purposes because we will
work with a fixed $a$. Let us also remark that the arguments given below can be applied to
obtain analogous interpolation error estimates in $L_a^p(K)$, $1\le p<\infty$,
provided $a\in A_p$ (see, for example, \cite{D} for the definition of these classes).

To simplify notation we will prove all the estimates for the weight $a$
although some of them will be used later for $a^{-1}$. Note that, from the definition of $A_2$, it follows immediately that $a\in A_2$ if and only if $a^{-1}\in A_2$.

In what follows we will use the following observation: under the regularity assumption
it is easy to see that
\begin{equation}
\label{A2 en K}
\left(\frac{1}{|K|}\int_K a\right) \left(\frac{1}{|K|}\int_K a^{-1}\right)\le C [a]_{A_2}
\end{equation}
with $C$ depending only on $n$ and $\eta$.

\begin{lemma} For $a \in A_2$ there exists a constant
depending only on $a$ and $\eta$, such that,
\begin{equation}
\label{cota Phu}
\|P_Ku\|_{L^2_{a}(K)} \le C\|u\|_{L^2_{a}(K)}
\end{equation}
and
\begin{equation}
\label{poincare clasica}
\|u-P_Ku\|_{L^2_{a}(K)} \le C h_K\|\nabla u\|_{L^2_{a}(K)}.
\end{equation}
\end{lemma}
\begin{proof}
We have
$$
\begin{aligned}
\|P_Ku\|_{L^2_{a}(K)}&=\left|\frac1{|K|}\int_K u\right|\left(\int_K a\right)^{1/2}\\
&\le\left(\frac1{|K|}\int_K a\right)^{1/2}\left(\frac1{|K|}\int_K a^{-1}\right)^{1/2} \|u\|_{L_a^2(K)},
\end{aligned}
$$
where we have used the Schwarz inequality in the last step. Therefore, \eqref{cota Phu} follows from \eqref{A2 en K}.

On the other hand, \eqref{poincare clasica} is the well known weighted Poincar\'e
inequality.
It was first proved in \cite{FKS} for the case of a ball and extended for
very general domains in several papers (see, for example, \cite{Ch,DL,H-S}).
The dependence of the constant on $h_T$ can be obtain by usual scaling arguments.
\end{proof}

Observe that, since $a\in A_2$ then $a^{-1}\in L^1(K)$ and, therefore,
$$
\|v\|_{L^1(K)}\le \left(\int_K a^{-1}\right)^{1/2} \|v\|_{L_a^2(K)}.
$$
Consequently $H^1_a(K)\subset W^{1,1}(K)$ and, in particular, traces of functions
in $H^1_a(K)$ on a face $F$ are well defined and belong to $L^1(F)$.

Our error estimates are based on the following generalized Poincar\'e inequality.

\begin{lemma}
\label{poincare generalizada}
Given $a\in A_2$, a simplex or a rectangle $K$ and a face $F$ of $K$, there exists a constant $C$,
depending only on $a$ and the regularity constant $\eta$, such that
$$
\|v-v_F\|_{L^2_a(K)}
\le C h_K \|\nabla v\|_{L^2_a(K)}
$$
for all $v\in H^1_a(K)$.
\end{lemma}
\begin{proof} In view of \eqref{poincare clasica} it is enough to estimate
    $\|v_K-v_F\|_{L^2_a(K)}$.

As we have mentioned above $H^1_a(K)\subset W^{1,1}(K)$, in particular $v_F$ is well defined.
Writing now
$$
v_F - v_K = \frac1{|F|}\int_F (v-v_K),
$$
and using the classic trace theorem
$$
\|v-v_K\|_{L^1(F)}
\le C \frac{|F|}{|K|}
\Big\{\|v-v_K\|_{L^1(K)}+ h_K \|\nabla v\|_{L^1(K)}\Big\}
$$
combined with the classic Poincar\'e inequality in $L^1(K)$, we obtain
$$
|v_F - v_K|
\le C \frac{h_K}{|K|}\|\nabla v\|_{L^1(K)}
$$
and consequently,
$$
\begin{aligned}
\|v_K-v_F\|_{L^2_a(K)}
&\le \left(\int_K a\right)^{1/2} |v_K - v_F|
\le C\left(\int_K a\right)^{1/2}
\frac{h_K}{|K|}\|\nabla v\|_{L^1(K)}\\
&\le C \left(\frac{1}{|K|}\int_K a\right)^{1/2}
\left(\frac{1}{|K|}\int_K a^{-1}\right)^{1/2} h_K\|\nabla v\|_{L_a^2(K)},
\end{aligned}
$$
and so, in view of \eqref{A2 en K}, the lemma is proved. \end{proof}

We can now prove the error estimates for the Raviart-Thomas interpolation.
We will denote with $D\bs$ the differential matrix of $\bs$.

\begin{lemma}
Given $a\in A_2$ and a simplex or a rectangle $K$ there exists a constant depending only on $a$
and the regularity constant $\eta$, such that
\begin{equation}
\label{RT error en regulares}
\|\bs-\Pi_K\bs\|_{L^2_a(K)}
\le C h_K \|D\bs\|_{L^2_a(K)}.
\end{equation}
\end{lemma}
\begin{proof}
Consider first the case of a simplex. We choose three faces $F_i$ with
corresponding normals $\bn_i$.

From \eqref{def de pi} we have
$\int_{F_i}(\bs-\Pi_K\bs)\cdot\bn_i=0$ and, therefore, using Lemma \ref{poincare generalizada}
we obtain,
\begin{equation}
\label{RT error en regulares 2}
\|(\bs-\Pi_K\bs)\cdot\bn_i\|_{L^2_a(K)}
\le C h_K \|\nabla[(\bs-\Pi_K\bs)\cdot\bn_i]\|_{L^2_a(K)}.
\end{equation}
But, $\frac{\partial(\Pi_K\bs)_j}{\partial x_k}=0$ for $k\neq j$, while
$$
\frac{\partial(\Pi_K\bs)_j}{\partial x_j}=\frac{\mbox{div\,}\Pi_K\bs}{n}
=\frac{P_K\mbox{div\,}\bs}{n},
$$
where we have used the commutative diagram property. Then, in view of \eqref{cota Phu}
we obtain
\begin{equation}
\label{RT error en regulares 3}
\|\nabla(\Pi_K\bs\cdot\bn_i)\|_{L^2_a(K)}
\le C\|D\bs\|_{L^2_a(K)}.
\end{equation}
Now observe that,
$$
|\bs-\Pi_K\bs|\le C\sum_{i=1}^3|(\bs-\Pi\bs)\cdot\bn_i|
$$
with a constant depending only on $\eta$, and so, \eqref{RT error en regulares} follows
from \eqref{RT error en regulares 2}
and \eqref{RT error en regulares 3}.

For a rectangular element we proceed in the same way. The only difference is that
to prove \eqref{RT error en regulares 3} we use now \eqref{cota Phu} combined with
$$
\frac{\partial(\Pi_K\bs)_j}{\partial x_j}
=P_K\left(\frac{\partial\bs_j}{\partial x_j}\right).
$$
\end{proof}

Combining the error estimates obtained above with the results of the previous section
we can now state the main theorem for approximation by Raviart-Thomas of lowest order
on regular families of meshes.

\begin{theorem}
\label{teorema en mallas regulares}
Let $\cth$ be a family of meshes with regularity constant $\eta$ and $h=\max_{K\in\cth}h_K$.
If $(\bs,u)$ and $(\bs_h,u_h)$ are the solutions of (\ref{Neumann}) and (\ref{mf}),
and (\ref{discrete Neumann}) and (\ref{mixedfem}) respectively then, for $a\in A_2$,
there exists a constant $C$ depending only on ${\mathcal D}$, $a$ and $\eta$ such that
$$
\|\bs-\bs_h\|_{L_{a^{-1}}^2}
\le Ch \|D\bs\|_{L_{a^{-1}}^2}
$$
and
$$
\|u-u_h\|_{L_a^2}
\le Ch \left\{\|D\bs\|_{L_{a^{-1}}^2}+\|\nabla u\|_{L_a^2}\right\}.
$$
\end{theorem}

\begin{proof}
The error estimate for $\bs$ follows from Lemma \ref{error sigma} combined with
the estimate \eqref{RT error en regulares} applied to the weight $a^{-1}$ (recall that
$a\in A_2$ if and only if $a^{-1}\in A_2$).

On the other hand, observe that
\eqref{RT error en regulares} implies the hypothesis \eqref{prop4} assumed in Lemma \ref{error u}.
Then, to bound the error for $u$ we apply that lemma,
\eqref{RT error en regulares} again, and \eqref{poincare clasica}.
\end{proof}

\section{Anisotropic error estimates}
\label{RT anisotropicos}
\setcounter{equation}{0}
Our next goal is to prove anisotropic error estimates suitable for problems
with boundary layers. For this kind of problems it is useful to have
estimates involving a weighted norm on the right hand side where the weight
is a power of the distance to some part of the boundary.

To present the main arguments we consider first the case of rectangular elements.
Then we show how similar ideas can be applied to prismatic elements which are
of interest in the application that we are going to consider in the next section, and more generally,
in many problems with solutions presenting boundary layers.
The case of simplex can be treated in a similar way but, as in the un-weighted case,
anisotropic error estimates are valid
only for some particular kind of degenerate elements (see \cite{AD1}).

Proceeding as in the previous section, we need now the following weighted improved Poincar\'e inequality, which is well known (see, for example, \cite{H1,DD}).
For $a\in A_2$ and $Q$ a cube,
\begin{equation}
\label{promedio cero}
\|v-v_Q\|_{L^2_a(Q)}\le C\|d\nabla v\|_{L^2_a(Q)}
\end{equation}
where $d$ denotes the distance to $\partial Q$.
Consider an arbitrary rectangle
$$
R=[a_1,b_1]\times\cdots\times[a_n,b_n].
$$
If we replace $Q$ by $R$ in the above inequality, it is known that
the constant in (\ref{promedio cero}) blows up when the ratio
between outer and inner diameter goes to infinity. However, we have the following
anisotropic version if the weight belongs to the smaller class
$A^s_2$ defined in the introduction.
For $i=1,\cdots,n$ we define
$$
d_i(x)=\min\{(b_i-x_i),(x_i-a_i)\} \qquad \mbox{and } \qquad h_i=b_i-a_i.
$$
\begin{lemma}
\label{lema 3 1}
For $a\in A^s_2$,
\begin{equation}
\label{promedio cero anisotropica}
\|v-v_R\|_{L^2_a(R)}
\le C\sum_{i=1}^n\left\|d_i\frac{\partial v}{\partial x_i}\right\|_{L^2_a(R)}.
\end{equation}
\end{lemma}
\begin{proof} It follows immediately from (\ref{promedio cero}) that,
if $Q$ is the unitary cube
$$
\|v-v_Q\|_{L^2_a(Q)}
\le C\sum_{i=1}^n\left\|d_i\frac{\partial v}{\partial x_i}\right\|_{L^2_a(Q)}.
$$
Then, (\ref{promedio cero anisotropica}) follows by standard arguments making
the change of variables $x_i=h_i\hat x_i+a_i$ and
using that, for $\hat a(\hat x):=a(x)$,
$a\in A_2^s \Longrightarrow\hat a\in A_2^s$.
\end{proof}

\begin{lemma}
\label{sobre caras}
For $a\in A^s_2$ and $F$ the face contained in $x_j=a_j$ we have
\begin{equation}
\label{promedio cero en cara}
\|v-v_F\|_{L^2_a(R)}
\le C\left\{\left\|(b_j-x_j)\frac{\partial v}{\partial x_j}\right\|_{L^2_a(R)}
+\sum_{i\neq j}\left\|d_i\frac{\partial v}{\partial x_i}\right\|_{L^2_a(R)}\right\}.
\end{equation}
\end{lemma}
\begin{proof}
By a simple integration by parts in the $x_j$ variable we have
$$
\frac1{|F|}\int_F v\,dS
=\frac1{|R|}\int_R v\,dx
+ \frac1{|R|}\int_R (x_j-b_j) \frac{\partial v}{\partial x_j}\,dx
$$
Then,
$$
v-v_F=v-v_R
-\frac1{|R|}\int_R (x_j-b_j) \frac{\partial v}{\partial x_j}\,dx
$$
and therefore,
$$
\|v-v_F\|_{L^2_a(R)}\le \|v-v_R\|_{L^2_a(R)}
+\frac1{|R|}\left(\int_R a\, dx\right)^{1/2}
\int_R (b_j-x_j) \left|\frac{\partial v}{\partial x_j}\right|\,dx
$$
but, multiplying and dividing by $a^{1/2}$ and using the Schwarz
inequality we obtain
$$
\int_R (b_j-x_j) \left|\frac{\partial v}{\partial x_j}\right|\,dx
\le
\left\|(b_j-x_j) \frac{\partial v}{\partial x_j}\right\|_{L^2_a(R)}
\left(\int_R a^{-1}\, dx\right)^{1/2}
$$
and consequently,
$$
\|v-v_F\|_{L^2_a(R)}\le \|v-v_R\|_{L^2_a(R)}
+[a]^{1/2}_{A^s_2}\left\|(b_j-x_j) \frac{\partial v}{\partial x_j}\right\|_{L^2_a(R)}.
$$
Therefore, (\ref{promedio cero en cara}) follows from
(\ref{promedio cero anisotropica}).
\end{proof}

We can now prove anisotropic error estimates for the Raviart-Thomas interpolation
$\Pi_R$. Observe that each component $(\Pi_R\bs)_j$ depends only on $\bs_j$, and so,
to simplify notation we will write simply $\Pi_R\bs_j$.

\begin{lemma}
\label{errorinterp}
For $a\in A^s_2$ and $1\le j\le n$,
\begin{equation}
\label{RT error}
\|\bs_j-\Pi_R\bs_j\|_{L^2_a(R)}
\le C
\left\{\sum_{i\neq j}\left\|d_i\frac{\partial\bs_j}{\partial x_i}\right\|_{L^2_a(R)}
+h_j\left\|\frac{\partial\bs_j}{\partial x_j}\right\|_{L^2_a(R)}\right\}.
\end{equation}
\end{lemma}
\begin{proof} Since $\bs_j-\Pi_R\bs_j$ has vanishing mean value on the face
defined by $x_j=a_j$ we obtain from (\ref{promedio cero en cara}),
$$
\begin{aligned}
&\|\bs_j-\Pi_R\bs_j\|_{L^2_a(R)}\\
&\le C
\left\{\sum_{i\neq j}\left\|d_i\frac{\partial}{\partial x_i}(\bs_j-\Pi_R\bs_j)\right\|_{L^2_a(R)}
+h_j\left\|\frac{\partial}{\partial x_j}(\bs_j-\Pi_R\bs_j)\right\|_{L^2_a(R)}\right\}.
\end{aligned}
$$
But, for $i\neq j$, $\frac{\partial(\Pi_R\bs)_j}{\partial x_i}=0$.
On the other hand from the definition of $\Pi_R$ we have
$$
\frac{\partial(\Pi_R\bs)_j}{\partial x_j}=P_R\left(\frac{\partial\bs_j}{\partial x_j}\right)
$$
and a simple argument using the Schwarz inequality shows that, for any $v\in L^2_a(R)$,
$$
\|P_R v\|_{L^2_a(R)}\le [a]^{1/2}_{A^s_2}\|v\|_{L^2_a(R)}
$$
and therefore,
$$
\left\|\frac{\partial(\Pi_R\bs)_j}{\partial x_j}\right\|_{L^2_a(R)}
\le [a]^{1/2}_{A^s_2}\left\|\frac{\partial\bs_j}{\partial x_j}\right\|_{L^2_a(R)},
$$
and the lemma is proved.
\end{proof}

Now we analyze the case of prismatic elements. For notational convenience we
work in ${\mathbb R}^{n+1}$ and introduce the variables $(x,y)$, with $x=(x_1,\cdots,x_n)\in{\mathbb R}^n$
and $y\in{\mathbb R}$. Therefore, the class $A^s_2$ denotes now the class of weights
satisfying
$$
[a]_{A^s_2}:=\sup_{R} \left(\frac{1}{|R|}
\int_R a \right)\left(\frac{1}{|R|}
\int_R a^{-1}\right)<\infty,
$$
where the supremum is taken over all $n+1$-dimensional rectangles.

We consider elements $P=K\times[y_0,y_1]$ where $K$ is
an $n$-dimensional simplex and $y_j\in{\mathbb R}$ for $j=0,1$.

Similar arguments than those used above for the anisotropic estimates in rectangular elements can be
used in this case. To simplify notation we will prove
only the particular weighted estimates that we will need for the application considered
in the next section. We will denote by $h_K$ the diameter of $K$. The elements
considered are anisotropic because no relation between $h_K$ and $y_1-y_0$
is required. On the other hand, for the simplices we assume the regularity condition
$h_K/\rho_K\le\eta$.

\begin{lemma}
\label{sobre caras en prismas}
Given $a\in A^s_2$, $P=K\times[y_0,y_1]$ a prismatic element, and $F_P$ a face of $P$
given by $F_P:=F\times[y_0,y_1]$, where $F$ is a face of $K$, we have
\begin{equation}
\label{promedio cero en cara de K por I}
\|v-v_{F_P}\|_{L^2_a(P)}
\le C\left\{\left\|(y-y_0)\frac{\partial v}{\partial y}\right\|_{L^2_a(P)}
+h_K\left\|\nabla_x v\right\|_{L^2_a(P)}\right\}.
\end{equation}
\end{lemma}
\begin{proof}
Proceeding as in the proof of \eqref{promedio cero anisotropica} we can prove the
Poincar\'e type inequality
\begin{equation}
\label{Poincare anisotropica en prismas}
\|v-v_P\|_{L^2_a(P)}
\le C\left\{\left\|(y-y_0)\frac{\partial v}{\partial y}\right\|_{L^2_a(P)}
+h_K\left\|\nabla_x v\right\|_{L^2_a(P)}\right\}.
\end{equation}
We will denote with $dS_F$ and $dS_{F_P}$ the surface measures on $F$ and $F_P$
respectively. Calling $x_F$ the vertex of $K$ opposite to $F$ and
integrating by parts we have,
$$
\int_K (x-x_F)\cdot\nabla_x v\,dx=-n\int_K v\,dx + \int_F (x-x_F)\cdot\bn_F v\,dS_F
$$
but, for $x\in F$, $(x-x_F)\cdot\bn_F=n|K|/|F|$, and therefore,
$$
\frac1{|F|}\int_F v\,dS=\frac1{|K|}\int_K v\,dx+\frac1{n|K|}\int_K (x-x_F)\cdot\nabla_x v\,dx.
$$
Then, integrating in the variable $y$,
$$
\frac1{|F|}\int_{F_P} v\,dS_{F_P}=\frac1{|K|}\int_P v\,dxdy
+\frac1{n|K|}\int_P (x-x_F)\cdot\nabla_x v\,dxdy
$$
and dividing this equation by $(y_1-y_0)$ we obtain
$$
v-v_{F_P}=v-v_P
-\frac1{n|P|}\int_P (x-x_F)\cdot\nabla_x v\,dxdy
$$
which, using \eqref{Poincare anisotropica en prismas} and proceeding as in the last part
of the proof of Lemma \ref{sobre caras}, implies \eqref{promedio cero en cara de K por I}.
\end{proof}

\begin{lemma}
\label{sobre caras de arriba y abajo en prismas}
Given $a\in A^s_2$, $P=K\times[y_0,y_1]$ a prismatic element, and $F_P$ a face of $P$
given by $F_P:=K\times\{y_j\}$, $j=0$ or $1$, we have
\begin{equation}
\label{promedio cero en cara de K por I segunda}
\|v-v_{F_P}\|_{L^2_a(P)}
\le C\left\{(y_1-y_0)\left\|\frac{\partial v}{\partial y}\right\|_{L^2_a(P)}
+h_K\left\|\nabla_x v\right\|_{L^2_a(P)}\right\}.
\end{equation}
\end{lemma}
\begin{proof}
It is analogous to the proof of Lemma \ref{sobre caras}.
\end{proof}

The local Raviart-Thomas space for $P=K\times[y_0,y_1]$ is given by
$$
\RT_0(P)=\{\bt\,:\,
\bt(x)=(a_1+bx_1,\cdots,a_n+bx_n,a_{n+1}+cy)\ \ \mbox{with} \ \ a_i,b,c\in{\mathbb R}\}.
$$
Given a vector field $\bs$ we define $\bst=(\sigma_1,\cdots,\sigma_n)$
and write $\bs=(\bst,\sigma_{n+1})$.
Since the normals to the top and bottom faces of $P$ are orthogonal to the other ones,
the Raviart-Thomas interpolation can be written as
$$
\Pi_P\bs=(\Pi_K\bst,\Pi_{n+1}\sigma_{n+1})
$$
where $\Pi_K$ and $\Pi_{n+1}$ depend on $\bst$ and $\sigma_{n+1}$ respectively.
Indeed, they are defined by
\begin{equation*}
\label{def de pi en prismas 1}
\int_{F\times[y_0,y_1]} \Pi_K\bst\cdot\bn_F =\int_{F\times[y_0,y_1]}\bst\cdot\bn_F
\end{equation*}
for all face $F$ of $K$ and
\begin{equation*}
\label{def de pi en prismas 2}
\int_{K\times\{y_j\}} \Pi_{n+1}\sigma_{n+1}=\int_{K\times\{y_j\}}\sigma_{n+1}
\end{equation*}
for $j=1,2$.

\begin{lemma}
\label{errorinterp prismas}
For $a\in A^s_2$ and $P=K\times[y_0,y_1]$, we have
\begin{equation}
\label{RT error prismas}
\|\bst-\Pi_K\bst\|_{L^2_a(P)}
\le C\left\{\left\|(y-y_0)\frac{\partial\bst}{\partial y}\right\|_{L^2_a(P)}
+h_K\left\|D_x\bst\right\|_{L^2_a(P)}\right\}
\end{equation}
and
\begin{equation}
\label{RT error prismas 2}
\begin{aligned}
\|\sigma_{n+1}&-\Pi_{n+1}\sigma_{n+1}\|_{L^2_a(P)}\\
&\le C\left\{(y_1-y_0)\left\|\frac{\partial\sigma_{n+1}}{\partial y}\right\|_{L^2_a(P)}
+h_K\left\|\nabla_x\sigma_{n+1}\right\|_{L^2_a(P)}\right\}
\end{aligned}
\end{equation}
where $C$ depends only on $a$ and the regularity constant $\eta$.
\end{lemma}

\begin{proof} Since $(\bst-\Pi_K\bst)\cdot\bn_F$ has vanishing mean value
on $F_P=F\times[y_0,y_1]$ we can apply \eqref{promedio cero en cara de K por I}
to obtain
$$
\begin{aligned}
&\|(\bst-\Pi_K\bst)\cdot\bn_F\|_{L^2_a(P)}\\
&\le C\left\{\left\|(y-y_0)\frac{\partial}{\partial y}[(\bst-\Pi_K\bst)\cdot\bn_F]\right\|_{L^2_a(P)}
+h_K\left\|\nabla_x(\bst-\Pi_K\bst)\cdot \bn_F)\right\|_{L^2_a(P)}\right\},
\end{aligned}
$$
and using this estimate for $n$ different faces of $K$ together with the regularity
assumption, we arrive at
$$
\begin{aligned}
&\|\bst-\Pi_K\bst\|_{L^2_a(P)}\\
&\le C\left\{\left\|(y-y_0)\frac{\partial}{\partial y}(\bst-\Pi_K\bst)\right\|_{L^2_a(P)}
+h_K\left\|D_x(\bst-\Pi_K\bst)\right\|_{L^2_a(P)}\right\}.
\end{aligned}
$$
But $\frac{\partial(\Pi_K\bst)}{\partial y}=0$ and $\frac{\partial(\Pi_K\bst)_i}{\partial x_j}=0$
for $i\ne j$. On the other hand,
$\frac{\partial(\Pi_K\bst)_i}{\partial x_i}=\frac{\mbox{div\,}_x \Pi_K\bst}{n}$ and
$\mbox{div\,}_x \Pi_K\bst=\frac1{|P|} \int_P\mbox{div\,}_x\bst$, and so, a simple argument using the
Cauchy-Schwarz inequality yields,
$$
\left\|\mbox{div\,}_x\Pi_K\bst\right\|_{L^2_a(P)}
\le [a]^{1/2}_{A^s_2}\left\|\mbox{div\,}_x\bst\right\|_{L^2_a(P)}
$$
and puting all together we obtain \eqref{RT error prismas}.

The proof of \eqref{RT error prismas 2} is analogous using now that
$\sigma_{n+1}-\Pi_{n+1}\sigma_{n+1}$ has vanishing mean value on the
face $K\times\{y_0\}$, applying \eqref{promedio cero en cara de K por I segunda},
and using that $\nabla_x(\Pi_{n+1}\sigma_{n+1})=0$
and $\frac{\partial}{\partial y}(\Pi_{n+1}\sigma_{n+1})=\frac1{|P|} \int_P\frac{\partial\sigma_{n+1}}{\partial y}$.
\end{proof}

\section{Fractional Laplacian}
\label{fraccionario}
\setcounter{equation}{0}

As an interesting application of the general results for degenerate problems we
consider the fractional Laplace equation. Given $\Omega\subset{\mathbb R}^n$ and
$f\in L^2(\Omega)$ we want to solve
\begin{equation}
\label{frac laplacian}
\left\{
\begin{array}{rl}
(-\Delta)^s v=f&\ \ \mbox{in} \ \Omega\\
v=0 & \ \ \mbox{on} \ \ \Omega^c
\end{array}
\right.
\end{equation}
for $0<s<1$.

Caffarelli and Silvestre have shown that the solution of this problem can be obtained
as $v(x)=u(x,0)$ where $u(x,y)$ is the solution
of a degenerate elliptic problem in a cylindrical domain in $n+1$ variables, namely,
\begin{equation}
\label{CS1}
\left\{
\begin{array}{rl}
\mbox{div\,}(y^\alpha \nabla u(x,y))=0 &\ \ \mbox{in} \  \mathcal{C}=\Omega\times (0,\infty)\\
-\lim_{y\to 0} y^\alpha\frac{\partial u}{\partial y} =d_s f& \ \ \mbox{on} \ \ \Gamma_N= \Omega\times\{0\}\\
u=0 & \ \ \mbox{on} \ \ \Gamma_D= \partial \mathcal{C}\setminus \Gamma_N
\end{array}
\right.
\end{equation}
with $d_s=2^{1-2s} \frac{\Gamma(1-s)}{\Gamma(s)}$ and $\alpha=1-2s$.
To solve this equation numerically one has to approximate
the domain $\mathcal{C}$ by a bounded one. With this goal we consider
a problem analogous to \eqref{CS1} with $\mathcal{C}$ replaced by
$\mathcal{C}_L=\Omega\times (0,L)$ and adding a homogeneous Dirichlet boundary condition
on the upper boundary of $\mathcal{C}_L$, namely, we look for $u_L$ such that,
\begin{equation}
\label{CS2}
\left\{
\begin{array}{rl}
\mbox{div\,}(y^\alpha \nabla u_L(x,y))=0 &\ \ \mbox{in} \  \mathcal{C}_L=\Omega\times (0,L)\\
-\lim_{y\to 0} y^\alpha\frac{\partial u_L}{\partial y} =f& \ \ \mbox{on} \ \ \Gamma_N= \Omega\times\{0\}\\
u_L=0 & \ \ \mbox{on} \ \ \Gamma_D= \partial\mathcal{C}_L\setminus \Gamma_N
\end{array}
\right.
\end{equation}
We will use several results proved in \cite{NOS}, therefore, we recall some
notation used in that paper. For $0<s<1$, we denote $H^s(\Omega)$ the fractional Sobolev space
of order $s$. We define for $s\neq\frac12$, $\mathbb{H}^s(\Omega):=H_0^s(\Omega)$,
the closure of $C_0^{\infty}(\Omega)$ in $H^s(\Omega)$ and
$\mathbb{H}^{1/2}(\Omega):=H_{00}^{1/2}(\Omega)$, the interpolation space $[H_0^1(\Omega),L^2(\Omega)]_{1/2}$
obtained by the K-method (for details see \cite{L}).
$\mathbb{H}^{-s}(\Omega)$ denotes the dual space of $\mathbb{H}^s(\Omega)$ for $s\in (0,1)$.

For our error estimates we will need some a priori bounds for the
derivatives of the exact solution.

In \cite{NOS} the following a priori estimates for the solution
of problem \eqref{CS1} were proved,
\begin{equation}
\label{derivadas primeras u}
\|\nabla u\|_{L^2_{y^{\alpha}}(\mathcal{C})}
\le C \|f\|_{{\mathbb H}^{-s}(\Omega)}
\end{equation}
and, for $1\le i,j\le n$,
\begin{equation}
\label{cotas derivadas segundas}
\left\|\frac{\partial^2 u}{\partial x_i \partial x_j}\right\|_{L^2_{y^\alpha}(\mathcal{C})}
+
\left\|\frac{\partial^2 u}{\partial x_i \partial y}\right\|_{L^2_{y^\alpha}(\mathcal{C})}
\le C\|f\|_{\mathbb{H}^{1-s}(\Omega)}.
\end{equation}

We will use the following estimate: For $\gamma>-1$ and
$v\in L^1(\mathcal{C}_L)\cap L^2_{y^{\gamma}}(\mathcal{C}_L)$ such that  $\int_{\mathcal{C}_L} v=0$,
there exists a constant $C$ independent of $L$ such that,
\begin{equation}
\label{improved en CL}
\|v\|_{L^2_{y^{\gamma}}(\mathcal{C}_L)}
\le C \|y\nabla v\|_{L^2_{y^{\gamma}}(\mathcal{C}_L)}.
\end{equation}
This estimate can be proved using the arguments introduced in \cite{DD}.
Details of the proof are given in \cite[Lemma 2.2]{DLP2} for a square domain
but the arguments apply to more general domains, in particular to the cylindrical
ones considered here.
That the constant C does not depend on $L$ follows from the case $L=1$
combined with a standard scaling argument.

\begin{lemma}
\label{cotas sigma}
Let $u$ be the solution of \eqref{CS1} and $\bs=(\sigma_1,\cdots,\sigma_{n+1})=-y^\alpha \nabla u$.
Then, for $1\le j\le n$ and $1\le i\le n+1$,
\begin{equation}
\label{cota cruzadas ij}
\left\|\frac{\partial\sigma_i}{\partial x_j}\right\| _{L_{y^{-\alpha}}^2(\mathcal{C}_L)}
+\left\|\frac{\partial\sigma_{n+1}}{\partial y}
\right\| _{L^2_{y^{-\alpha}}(\mathcal{C}_L)}
\le C \|f\|_{\mathbb{H}^{1-s}(\Omega)},
\end{equation}
and for $1\le i \le n$ and $\beta>1-\alpha$,
\begin{equation}
\left\| \frac{\partial\sigma_i}{\partial y}
\right\| _{L^2_{y^{-\alpha + \beta}}(\mathcal{C}_L)}
\le C L^{\beta/2} \|f\|_{\mathbb{H}^{1-s}(\Omega)}.
\end{equation}
\end{lemma}
\begin{proof}
The bound for the first term in \eqref{cota cruzadas ij} follows immediately
from \eqref{cotas derivadas segundas}.
To estimate the second term observe that, from
(\ref{CS1}),
$$
\frac{ \partial\sigma_{n+1}}{\partial y}=-y^{\alpha}\Delta_xu
$$
and use \eqref{cotas derivadas segundas}.

For $1 \le i \le n$ we have
$$
\frac{ \partial\sigma_i}{\partial y}=-\alpha y^{\alpha-1}\frac{\partial u}{\partial x_i}
-y^{\alpha}\frac{\partial^2 u}{ \partial x_i \partial y}.
$$
To bound the second term
we use again \eqref{cotas derivadas segundas}.
For the first one we observe that  $\int_{\mathcal{C}_L}\frac{\partial u}{\partial x_i}=0$
because $u$ vanishes on $\partial\Omega\times (0,\infty)$, and therefore, since
$\beta>1-\alpha$ we can use
\eqref{improved en CL}  with $\gamma=\alpha-2+\beta$ to obtain
$$
\begin{aligned}
\int_{\mathcal{C}_L}\left| y^{\alpha-1}\frac{\partial u}{\partial x_i}\right| ^2 y^{-\alpha}y^{\beta}
&= \int_{\mathcal{C}_L}\left| \frac{\partial u}{\partial x_i}\right| ^2 y^{\alpha-2+\beta}
\le C \int_{\mathcal{C}_L} \left|\nabla\left( \frac{\partial u}{\partial x_i}\right) \right|^2
y^{\alpha+\beta}\\
&\le C L^{\beta}\int_{\mathcal{C}_L}\left|\nabla\left( \frac{\partial u}{\partial x_i}\right) \right|^2 y^{\alpha}
\le C L^{\beta}\|f\|^2_{\mathbb{H}^{1-s}(\Omega)}
\end{aligned}
$$
where we have used \eqref{cotas derivadas segundas} for the last inequality.\end{proof}

Our goal is to approximate $u$ and $\bs=-y^\alpha\nabla u$ given by \eqref{CS1}.
Since the problem is posed in the unbounded domain $\mathcal{C}$ we need to replace it
by $\mathcal{C}_L$ where $L$ will be chosen in terms of the mesh parameter $h$ in such a way
that $L\to\infty$ when $h\to 0$.

It was shown in \cite[Theorem 3.5]{NOS} that
for $f\in \mathbb{H}^{-s}(\Omega)$ and $L\ge 1$, if $u_L(x,y)$ is extended by zero
for $y>L$, there exists a constant $C$ such that
\begin{equation}
\label{exponencial}
\|\nabla(u-u_L)\|_{L_{y^\alpha}^2(\mathcal{C})}
\le C e^{-\sqrt{\lambda_1}L/4}\|f\|_{\mathbb{H}^{-s}(\Omega)}
\end{equation}
where $\lambda_1>0$ is the first eigenvalue of the Laplacian with Dirichlet boundary conditions in $\Omega$.

Moreover, using the Poincar\'e inequality
\begin{equation}
\label{poincare en el cilindro}
\|u-u_L\|_{L_{y^\alpha}^2(\mathcal{C})}
\le C \|\nabla(u-u_L)\|_{L_{y^\alpha}^2(\mathcal{C})},
\end{equation}
which follows easily applying the standard Poincar\'e inequality
in $\Omega$ for each $y$, multiplying by the weight, and integrating in $y$,
we also have
\begin{equation}
\label{exponencial 2}
\|u-u_L\|_{H_{y^\alpha}^1(\mathcal{C})}
\le C e^{-\sqrt{\lambda_1}L/4}\|f\|_{\mathbb{H}^{-s}(\Omega)}.
\end{equation}

Now we consider the mixed finite element approximation of \eqref{CS2}.
We will apply the results of the previous sections for ${\mathcal D}=\mathcal{C}_L$ and $\Gamma_N=\Omega\times\{0\}$.
However, since we want error estimates in terms of $\bs$ instead of $\bs_L$, to take advantage of the known
a priori estimates, we need to introduce some minor modifications in the error analysis.

Given a family of meshes $\cth$ made by prismatic elements as those considered in
the last part of Section \ref{RT anisotropicos} and the associated spaces
${\pmb S}_h$ and $V_h$ defined as in \eqref{Raviart-Thomas} and \eqref{Constantes a trozos},
the approximate solutions $u_{L,h}\in V_h$
and $\bs_{L,h} \in{\pmb S}_h$ are given by,
\begin{equation}
\label{condicion de Neumann debil truncado}
\bs_{L,h}\cdot\bn|_F=\frac1{|F|}\int_F f ,
\end{equation}
for every face $F$ contained in $\Omega$,
and
\begin{equation}
\label{debil truncado}
\left\{
\begin{aligned}
\int_{\mathcal{C}_L} y^{-\alpha}\,\bs_{L,h}\cdot\bt - \int_{\Omega} u_{L,h}\, \mbox{div\,}\bt
= 0
&\qquad\qquad\quad \forall \bt \in {\pmb S}_{h,N} \\
\int_{\mathcal{C}_L} v\,\mbox{div\,}\bs_{L,h}
= 0
&\qquad\qquad\quad\forall v\in V_h
\end{aligned}
\right.
\end{equation}
where ${\pmb S}_{h,N}:={\pmb S}_h\cap\hdivN$.

\begin{theorem}
\label{sigma L} Let $u$ and $u_L$ be the solutions of (\ref{CS1}) and (\ref{CS2})
respectively, $\bs=-y^\alpha\nabla u$ and $\bs_L=-y^\alpha\nabla u_L$.
If $u_{L,h}$ and $\bs_{L,h}$ are the approximate solutions given by \eqref{debil truncado}, then
\begin{equation}
\label{cota sigma menos sigma Lh}
\|\bs-\bs_{L,h}\|_{L_{y^{-\alpha}}^2(\mathcal{C}_L)}
\le
\|\bs-\Pi_h \bs\|_{L_{y^{-\alpha}}^2(\mathcal{C}_L)}+ \|\bs-\bs_L\|_{L_{y^{-\alpha}}^2(\mathcal{C}_L)},
\end{equation}
and
\begin{equation}
\begin{aligned}
\label{cota u menos u Lh}
\|u-u_{L,h}\|_{L_{y^{\alpha}}^2(\mathcal{C}_L)}
&\le C\|u-P_hu\|_{L_{y^{\alpha}}^2(\mathcal{C}_L)}\\
&+CL\Big\{\|\bs-\Pi_h \bs\|_{L_{y^{-\alpha}}^2(\mathcal{C}_L)}
+ \|\bs-\bs_L\|_{L_{y^{-\alpha}}^2(\mathcal{C}_L)}\Big\}.
\end{aligned}
\end{equation}
\end{theorem}
\begin{proof}
Observing that $\Pi_h\bs-\bs_{L,h}\in{\pmb S}_{h,N}$ and $\mbox{div\,}(\Pi_h\bs-\bs_{L,h})=0$
and proceeding as in the proof of Lemma \ref{error sigma} we obtain,
$$
\int_{\mathcal{C}_L} y^{-\alpha} (\bs_L-\bs_{L,h})\cdot(\Pi_h \bs- \bs_{L,h})=0.
$$
Then,
$$
\|\bs_L- \bs_{L,h}\|^2_{L_{y^{-\alpha}}^2(\mathcal{C}_L)}
=\int_{\mathcal{C}_L} y^{-\alpha} (\bs_L- \bs_{L,h})\cdot(\bs_L-\Pi_h\bs),
$$
and therefore,
\begin{equation}
\label{sigmaL-sigmaLh}
\|\bs_L- \bs_{L,h}\|_{L_{y^{-\alpha}}(\mathcal{C}_L)}
\le\|\bs_L-\Pi_h\bs\|_{L_{y^{-\alpha}}(\mathcal{C}_L)},
\end{equation}
which combined with a triangular inequality yields \eqref{cota sigma menos sigma Lh}.

On the other hand, for our domain $\mathcal{C}_L$
the inequality from Lemma \ref{error u} can be written as
\begin{equation}
\label{uL-uLh}
\|u_L-u_{L,h}\|_{L_{y^{\alpha}}^2(\mathcal{C}_L)}
\le \|u_L-P_hu_L\|_{L_{y^{\alpha}}^2(\mathcal{C}_L)}
+ C L\|\bs_L-\bs_{L,h}\|_{L_{y^{-\alpha}}^2(\mathcal{C}_L)}
\end{equation}
where the constant $C$ is independent of $L$. Indeed, this follows
from the proof of that lemma once we know that the constant
in Lemma \ref{inversa de la div} is proportional to $L$, which follows
from the case $L=1$ and a scaling argument.

To bound the second term in the right hand side of \eqref{uL-uLh}
we use \eqref{sigmaL-sigmaLh}, while for the first one we have
$$
\begin{aligned}
\|u_L-P_hu_L\|_{L_{y^{\alpha}}^2(\mathcal{C}_L)}
&\le \|u-P_hu\|_{L_{y^{\alpha}}^2(\mathcal{C}_L)}
+\|(u-u_L)-P_h(u-u_L)\|_{L_{y^{\alpha}}^2(\mathcal{C}_L)}\\
&\le \|u-P_hu\|_{L_{y^{\alpha}}^2(\mathcal{C}_L)}
+ C \|\nabla(u-u_L)\|_{L_{y^{\alpha}}^2(\mathcal{C}_L)}
\end{aligned}
$$
where in the last inequality we have used the version for prisms of
\eqref{promedio cero anisotropica}. To conclude the proof we observe that
$$
\|\nabla(u-u_L)\|_{L_{y^{\alpha}}^2(\mathcal{C}_L)}
=\|\bs-\bs_L\|_{L_{y^{-\alpha}}^2(\mathcal{C}_L)}
$$
and, therefore, from the Poincar\'e inequality \eqref{poincare en el cilindro} we obtain
$$
\|u-u_L\|_{L_{y^{\alpha}}^2(\mathcal{C}_L)}
\le C\|\bs-\bs_L\|_{L_{y^{-\alpha}}^2(\mathcal{C}_L)}.
$$
\end{proof}

Next we are going to show that introducing appropriate meshes, graded in the $y$-direction,
we obtain almost optimal order of convergence with respect to the number of nodes, i. e.,
the same order than that valid for problems with smooth solutions using uniform meshes,
up to a logarithmic factor.

Given a mesh-size $h>0$, to define $\cth$ we start with a quasi-uniform triangulation of $\Omega$ made of simplices
of diameter less than or equal to $h$. Then, for $L\ge 1$ to be chosen below in terms
of $h$, we introduce a partition of $[0,L]$ given by
\begin{equation}\label{malla_y}
y_j=\left(\frac{j}{N} \right)^{\frac2{2-\beta}} L, \qquad j=0,\cdots, N
\end{equation}
where $N\sim 1/h$ (we take $N=1/h$ if it is an integer or some approximation of it if not), and $\beta\in(1-\alpha,2)$ to be chosen (in the numerical experiments we have taken $\beta$ as the midpoint of this interval).
Finally, the partition $\cth$ of $\mathcal{C}_L$ is formed by the prismatic elements
$P=K\times[y_j,y_{j+1}]$, where $K$ are the elements in the partition of $\Omega$.

It follows from this definition that, for $j\ge 1$,
\begin{equation}
\label{yj+1-yj}
(y_{j+1}-y_j)^2\le C_\beta h^2 y_j^\beta L^{2-\beta},
\end{equation}
indeed, by the mean value theorem and using that $h\sim 1/N$
we have
$$
y_{j+1}-y_j\le C \frac{\beta}{2-\beta} (jh)^{\frac{\beta}{2-\beta}} h L
\le  C \frac{\beta}{2-\beta} y_j^{\frac{\beta}{2}} h L^{1-\frac{\beta}2}.
$$

Using the notation introduced for prismatic elements in the previous section,
the Raviart-Thomas interpolation is given by $\Pi_h\bs=(\tilde{\Pi}_h\bst,\Pi_{h,n+1}\sigma_{n+1})$
where $\tilde{\Pi}_h$ and $\Pi_{h,n+1}$ are given locally by $\Pi_K$ and $\Pi_{n+1}$
respectively. We recall that, since $-1<\alpha<1$, $y^{\alpha}$ and $y^{-\alpha}$ belong to $A_2^s$.

\begin{theorem}
\label{teorema 5.3}
For some $\beta\in(1-\alpha,2)$, consider the family of meshes $\cth$ defined above.
Let $u$ be the solution of (\ref{CS1}),
$\bs=-y^{\alpha}\nabla u$, and
$(u_{L,h}, \bs_{L,h})$ be the approximation given by \eqref{condicion de Neumann debil truncado}
and \eqref{debil truncado}. Then, if $L=C_1|\log h|$
with $C_1\ge 4/\sqrt{\lambda_1}$, we have
\begin{equation}
\label{error sigma definitivo}
\|\bs-\bs_{L,h}\|_{L_{y^{-\alpha}}^2(\mathcal{C}_L)} \le Ch |\log h| \|f\|_{\mathbb{H}^{1-s}(\Omega)},
\end{equation}
and
\begin{equation}
\label{error u definitivo}
\|u-u_{L,h}\|_{L_{y^{\alpha}}^2(\mathcal{C}_L)}\le C h |\log h|^2\,\|f\|_{\mathbb{H}^{1-s}(\Omega)},
\end{equation}
where the constant $C$ depends on $\Omega$, $\alpha$, and $\beta$.
\end{theorem}
\begin{proof}
From \eqref{cota sigma menos sigma Lh} and \eqref{exponencial} we have
\begin{equation}
\label{error de sigma 1}
\|\bs-\bs_{L,h}\|_{L_{y^{-\alpha}}^2(\mathcal{C}_L)}
\le \|\bs- \Pi_h\bs\|_{L_{y^{-\alpha}}^2(\mathcal{C}_L)}
+ C e^{-\sqrt{\lambda_1}L/4}\|f\|_{\mathbb{H}^{-s}(\Omega)}.
\end{equation}
Applying \eqref{RT error prismas} for the elements of the form  $P=K\times [0,y_1]$
and summing over all of them we obtain,
$$
\|\bst-\tilde\Pi_h\bst\|^2_{L_{y^{-\alpha}}^2(\Omega\times[0,y_1])}
\le
C\left\{h^2\|D_x\bst\|^2_{L_{y^{-\alpha}}^2(\Omega\times[0,y_1])}
+ \left\| y\frac{ \partial\bst}{\partial y}\right\|^2 _{L_{y^{-\alpha}}^2(\Omega\times[0,y_1])}\right\}.
$$
But,
$$
\begin{aligned}
\left\| y\frac{\partial\bst}{\partial y}\right\|^2 _{L_{y^{-\alpha}}^2(\Omega\times[0,y_1])}
&=\int_{\Omega}\int_0^{y_1} y^2
\left|\frac{\partial\bst}{\partial y}\right|^2 y^{-\alpha}dy dx\\
&\le y_1^{2-\beta}
\int_{\Omega}\int_0^{y_1}
\left|\frac{\partial\bst}{\partial y}\right|^2 y^{-\alpha+\beta}dy dx\\
&\le C h^2 L^{2-\beta}\int_{\Omega}\int_0^{y_1}
\left|\frac{\partial\bst}{\partial y}\right|^2 y^{-\alpha+\beta}dy dx
\end{aligned}
$$
where in the last inequality we have used the definition of $y_1$.
Then,
$$
\begin{aligned}
\|\bst&-\tilde\Pi_h\bst\|^2_{L_{y^{-\alpha}}^2(\Omega\times[0,y_1])}\\
&\le
C\left\{h^2\|D_x\bst\|^2_{L_{y^{-\alpha}}^2(\Omega\times[0,y_1])}
+ h^2 L^{2-\beta}\left\|\frac{ \partial\bst}{\partial y}\right\|^2 _{L_{y^{-\alpha+\beta}}^2(\Omega\times[0,y_1])}\right\}.
\end{aligned}
$$
Analogously, applying now \eqref{RT error prismas 2}, we have
$$
\begin{aligned}
\|\sigma_{n+1}&-\Pi_{h,n+1}\sigma_{n+1}\|^2_{L_{y^{-\alpha}}^2(\Omega\times[0,y_1])}\\
&\le C\left\{h^2\|\nabla_x\sigma_{n+1}\|^2_{L_{y^{-\alpha}}^2(\Omega\times[0,y_1])}
+ y^2_1\left\| \frac{ \partial\sigma_{n+1}}{\partial y}\right\|^2 _{L_{y^{-\alpha}}^2(\Omega\times[0,y_1])}\right\},
\end{aligned}
$$
and therefore, using again the definition of $y_1$, we obtain
$$
\begin{aligned}
\|\sigma_{n+1}&-\Pi_{h,n+1}\sigma_{n+1}\|^2_{L_{y^{-\alpha}}^2(\Omega\times[0,y_1])}\\
&\le C\left\{h^2\|\nabla_x\sigma_{n+1}\|^2_{L_{y^{-\alpha}}^2(\Omega\times[0,y_1])}
+ h^2 L^2\left\| \frac{ \partial\sigma_{n+1}}{\partial y}\right\|^2 _{L_{y^{-\alpha}}^2(\Omega\times[0,y_1])}\right\},
\end{aligned}
$$
Consequently, combining the estimates above, we conclude
\begin{equation}
\label{error en primera banda}
\begin{aligned}
&\|\bs- \Pi_h\bs\|^2_{L_{y^{-\alpha}}^2(\Omega\times[0,y_1])}
\\ & \le
C\left\{h^2\|D_x\bs\|^2_{L_{y^{-\alpha}}^2(\Omega\times[0,y_1])}
+ h^2 L^{2-\beta}\left\|\frac{\partial\bst}{\partial y}\right\|^2 _{L_{y^{-\alpha+\beta}}^2(\Omega\times[0,y_1])}\right.
\\ & \left. \qquad\qquad\qquad\qquad\qquad\qquad\quad + h^2 L^2\left\| \frac{ \partial\sigma_{n+1}}
{\partial y}\right\|^2 _{L_{y^{-\alpha}}^2(\Omega\times[0,y_1])}\right\}.
\end{aligned}
\end{equation}
Aplying now \eqref{RT error prismas}and \eqref{RT error prismas 2} for
the elements of the form $P=K\times[y_j,y_{j+1}]$, for each $j\ge 1$, and summing over these elements we obtain
$$
\begin{aligned}
\|\bs&- \Pi_h\bs\|^2_{L_{y^{-\alpha}}^2(\Omega\times[y_j,y_{j+1}])}\\
&\le
C\left\{h^2 \|D_x\bs\|^2_{L_{y^{-\alpha}}^2(\Omega\times[y_j,y_{j+1}])}
+ (y_{j+1}-y_j)^2\left\| \frac{ \partial\bs}{\partial y}\right\| _{L^2_{y^{-\alpha}}(\Omega\times[y_j,y_{j+1}])}^2\right\},
\end{aligned}
$$
and using \eqref{yj+1-yj},
$$
\begin{aligned}
&\|\bs- \Pi_h\bs\|^2_{L_{y^{-\alpha}}^2(\Omega\times[y_j,y_{j+1}])}\\
&\le
C\left\{h^2 \|D_x\bs\|^2_{L_{y^{-\alpha}}^2(\Omega\times[y_j,y_{j+1}])}
+ C_\beta h^2 y_j^\beta L^{2-\beta}\left\| \frac{\partial\bs}{\partial y}\right\| _{L^2_{y^{-\alpha}}
(\Omega\times[y_j,y_{j+1}])}^2\right\}
\\
&\le
C\left\{h^2 \|D_x\bs\|^2_{L_{y^{-\alpha}}^2(\Omega\times[y_j,y_{j+1}])}
+ C_\beta h^2 L^{2-\beta}\left\| \frac{ \partial\bs}{\partial y}\right\| _{L^2_{y^{-\alpha+\beta}}
(\Omega\times[y_j,y_{j+1}])}^2\right\},
\end{aligned}
$$
and then, observing that
$$
L^{2-\beta}\left\| \frac{\partial\sigma_{n+1}}{\partial y}\right\| _{L^2_{y^{-\alpha+\beta}}(\Omega\times[y_j,y_{j+1}])}^2
\le
L^2
\left\| \frac{ \partial\sigma_{n+1}}{\partial y}\right\|^2 _{L_{y^{-\alpha}}^2(\Omega\times[y_j,y_{j+1}])},
$$
summing over $j$, and combining this with \eqref{error en primera banda}, we obtain
\begin{equation}
\label{error total}
\begin{aligned}
&\|\bs- \Pi_h\bs\|^2_{L_{y^{-\alpha}}^2(\mathcal{C}_L)}
\\ & \le
C\left\{h^2\|D_x\bs\|^2_{L_{y^{-\alpha}}^2(\mathcal{C}_L)}
+ h^2 L^{2-\beta}\left\|\frac{ \partial\bst}{\partial y}\right\|^2 _{L_{y^{-\alpha+\beta}}^2(\mathcal{C}_L)}\right.
\\ & \left. \qquad\qquad\qquad\qquad\qquad\qquad\quad + h^2 L^2
\left\| \frac{ \partial\sigma_{n+1}}{\partial y}\right\|^2 _{L_{y^{-\alpha}}^2(\mathcal{C}_L)}\right\},
\end{aligned}
\end{equation}
where, here and in what follows, the constant $C$ depends on $C_\beta$.

Applying now Lemma \ref{cotas sigma} and the bound \eqref{error total} it follows from \eqref{error de sigma 1} that
$$
\|\bs-\bs_{L,h}\|_{L_{y^{-\alpha}}^2(\mathcal{C}_L)}
\le ChL\|f\|_{\mathbb{H}^{1-s}(\Omega)}
+ C e^{-\sqrt{\lambda_1}L/4}\|f\|_{\mathbb{H}^{-s}(\Omega)}.
$$
From the hypothesis on $C_1$ we have $e^{-\sqrt{\lambda_1}L/4}\le h$ and, therefore,
\eqref{error sigma definitivo} is proved.

In view of \eqref{cota u menos u Lh}, to finish the proof of \eqref{error u definitivo}
it is enough to show that
\begin{equation}
\label{cota u-Phu}
\|u-P_hu\|_{L_{y^{\alpha}}^2(\mathcal{C}_L)}
\le ChL\|f\|_{{\mathbb H}^{-s}(\Omega)}.
\end{equation}
Using \eqref{Poincare anisotropica en prismas} for elements
of the form $K\times[0,y_1]$ we obtain
$$
\begin{aligned}
\|u-& P_hu\|_{L_{y^{\alpha}}^2(\Omega\times[0,y_1])}\\
&\le C\left\{h^{\frac2{2-\beta}}L
\left\|\frac{\partial u}{\partial y}\right\|_{L_{y^{\alpha}}^2(\Omega\times[0,y_1])}
+h\|\nabla_x u\|_{L_{y^{\alpha}}^2(\Omega\times[0,y_1])}\right\}\\
&\le ChL \|\nabla u\|_{L_{y^{\alpha}}^2(\Omega\times[0,y_1])},
\end{aligned}
$$
because $2/(2-\beta)\ge 1$ and $L\ge 1$.

On the other hand, \eqref{Poincare anisotropica en prismas} and \eqref{yj+1-yj} yields
$$
\|u-P_hu\|_{L_{y^{\alpha}}^2(\Omega\times[y_1,L])}
\le ChL\|\nabla u\|_{L_{y^{\alpha}}^2(\Omega\times[y_1,L])}
$$
and, therefore, taking into account \eqref{derivadas primeras u}, \eqref{cota u-Phu} is proved.
\end{proof}
Now we give some numerical examples showing the asymptotic  behavior
of the error proved in Theorem \ref{teorema 5.3}.
We solve Problem \eqref{CS1} with $\Omega=(0,1)\times(0,1)$ and
$$
f(x_1,x_2)=(2\pi^2)^s\sin(\pi x_1)\sin(\pi x_2).
$$
Recall that $0<s<1$ and $\alpha=1-2s$.
In this case, assuming that $s\neq 1/2$ (i. e. $\alpha\neq 0$), the solution is given by
$$
u(x_1,x_2,y)=\frac{2^{1-s}}{\Gamma(s)}(\sqrt{2}\pi y)^s K_s(\sqrt{2}\pi y)\sin(\pi x_1)\sin(\pi x_2)
$$
where $K_s$ is a modified Bessel function of the second kind (see \cite{NOS}).

We use prismatic elements given by a uniform mesh of triangles in $\Omega$ and the
refinement given by \eqref{malla_y} in the $y$-direction. Observe that for
these meshes $h\sim (DOF)^{-1/3}$ where $DOF$ denotes the degrees of freedom.
Moreover, we choose
$L$ as in Theorem \ref{teorema 5.3} with $C_1=1$, i. e., $L=|\log h|$.

The next graphics show the order of the errors $\|\bs-\bs_{L,h}\|_{L_{y^{-\alpha}}^2(\mathcal{C}_L)}$
and $\|u-u_{L,h}\|_{L_{y^{\alpha}}^2(\mathcal{C}_L)}$ for several values of $\alpha$.

\begin{figure}[H]
  \centering
    \includegraphics[width=7cm]{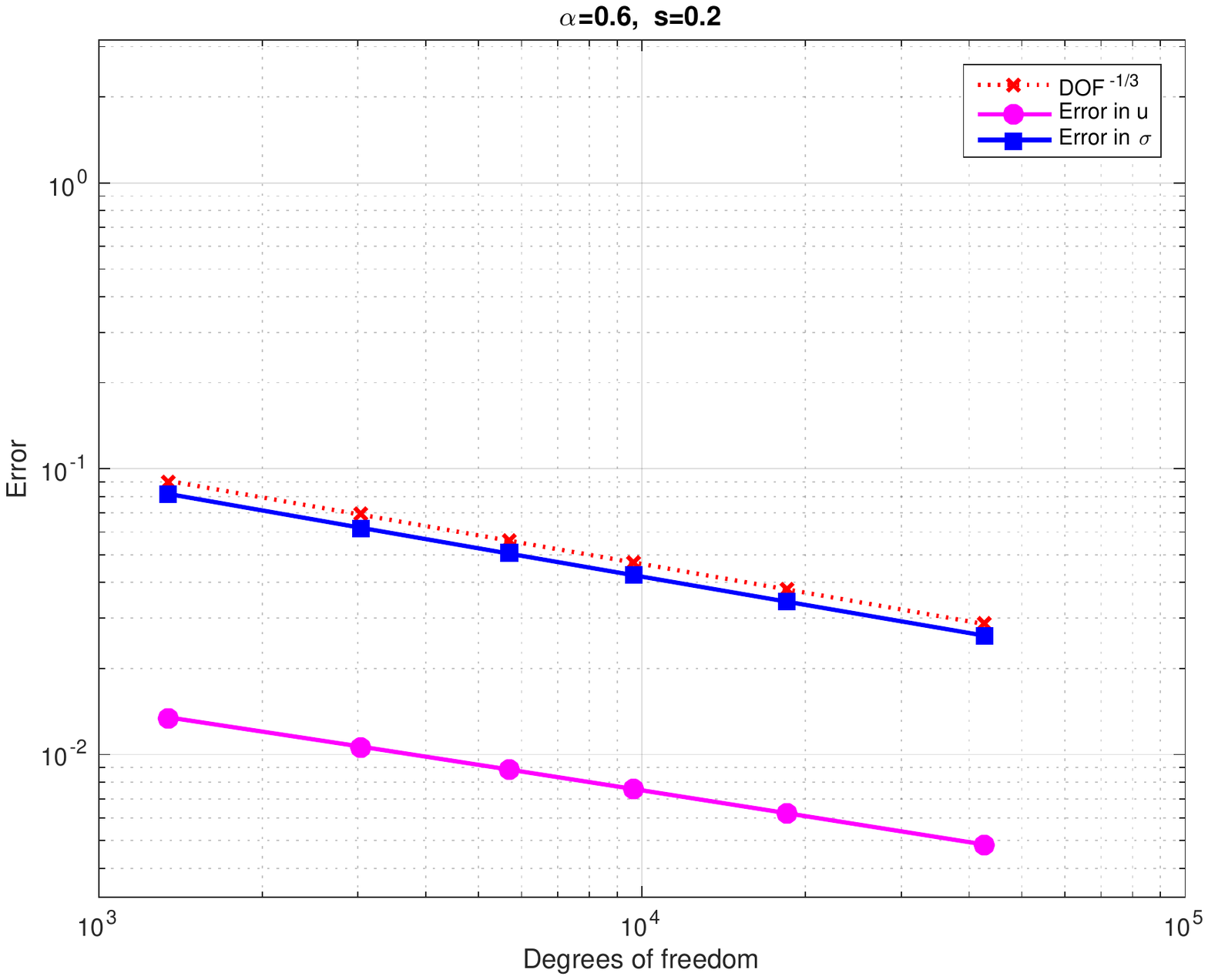}  \quad
    \includegraphics[width=7cm]{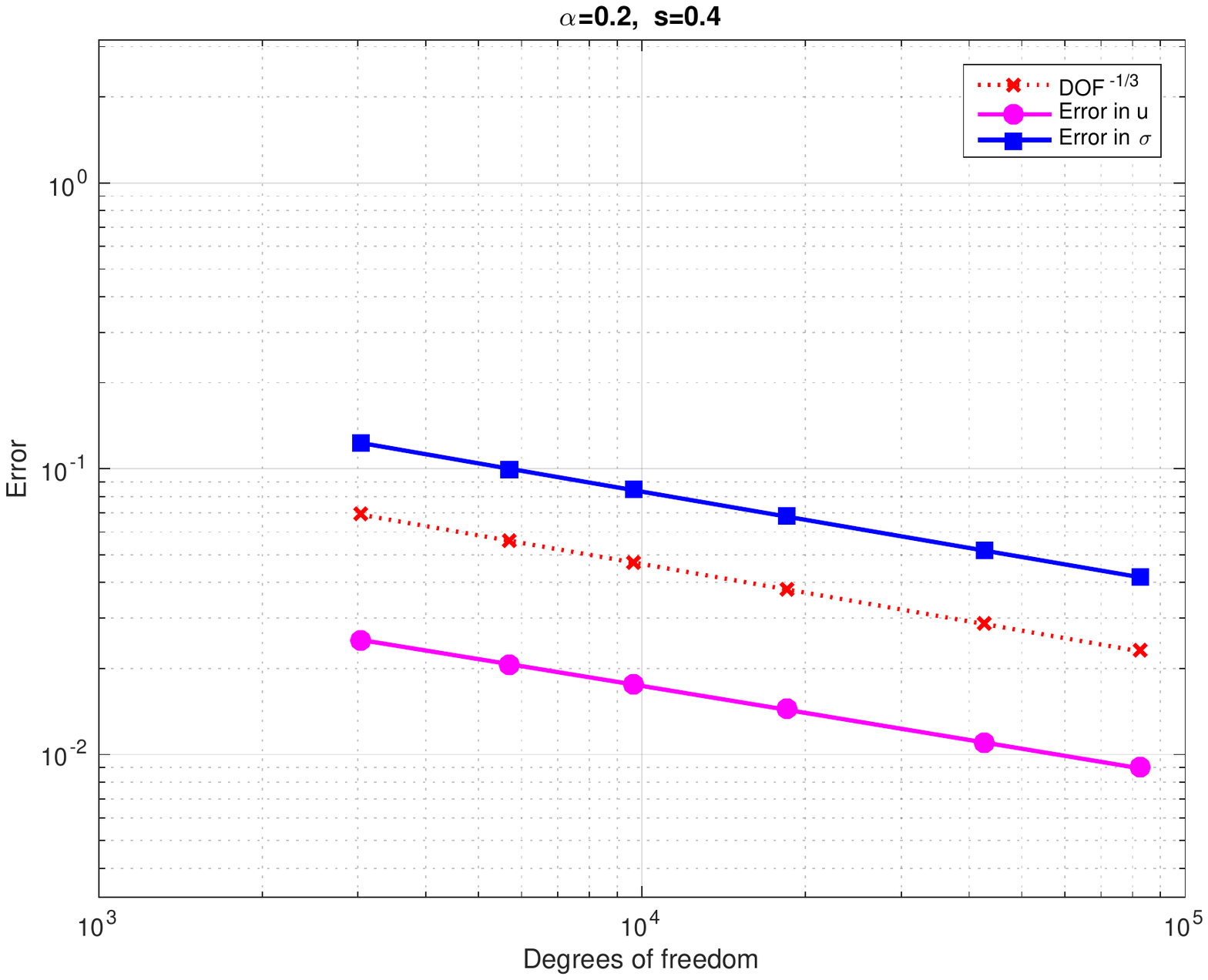}
  \caption{Rate of convergence: left $\alpha=0.6$, right $\alpha=0.2$.}
\end{figure}

\begin{figure}[H]
  \centering
    \includegraphics[width=7cm]{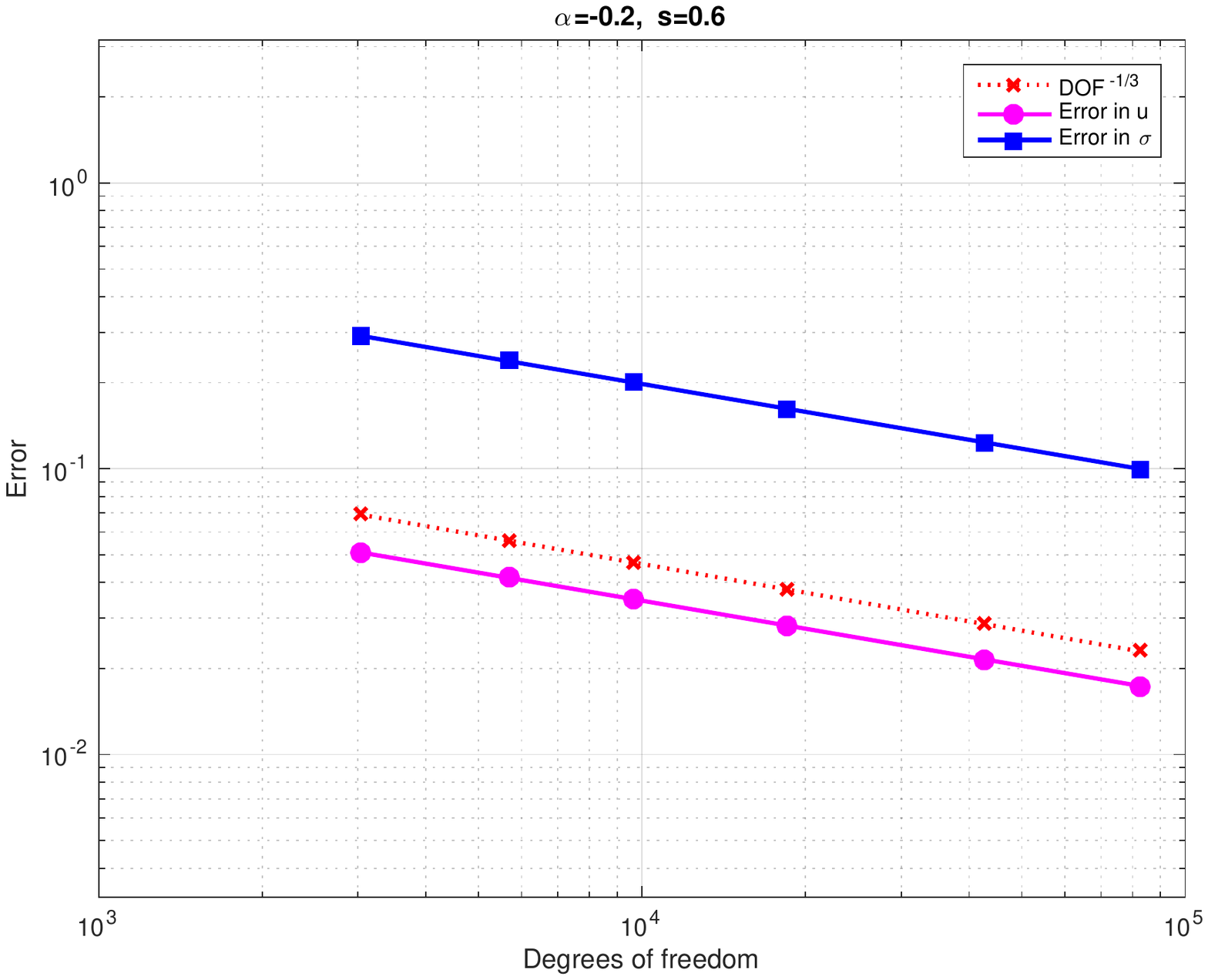}  \quad
    \includegraphics[width=7cm]{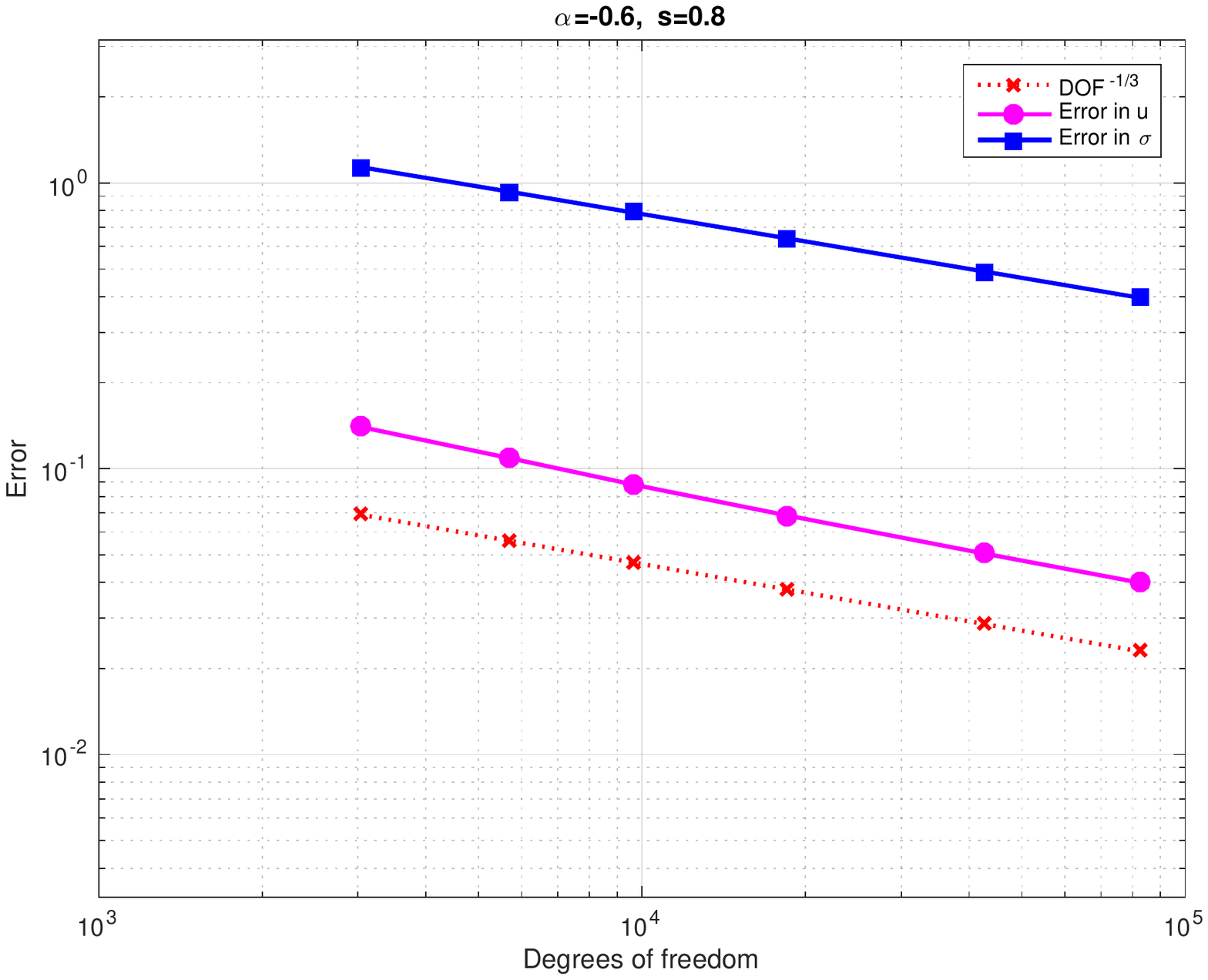}
  \caption{Rate of convergence: left $\alpha=-0.2$, right $\alpha=-0.6$.}
\end{figure}

Finally, to solve \eqref{frac laplacian}, we need to approximate $u(x,0)$ where $u$
is the solution of \eqref{CS1}.
We will use the approximations $u_{L,h}$ and $\bs_{L,h}$ obtained above.

Since $u_{L,h}$ is only an approximation in the $L^2$-norm, one cannot expect that its
restriction to $y=0$ be a good approximation of $u(x,0)$. In order to obtain a better
approximation we will make a local correction of $u_{L,h}$ using also the computed $\bs_{L,h}$.
This correction corresponds to a first order Taylor expansion,
indeed, the formula that we are going to prove in the next lemma is motivated by
$$
u(x,0)\sim u(x,\frac{y_1}2)-\frac{y_1}2 \,\frac{\partial u}{\partial y}\left(x,\frac{y_1}2\right).
$$
We will prove that in this way we obtain an approximation in $L^2(\Omega)$ of at least the same order
than the mixed finite element approximation of \eqref{CS1}.

Given $x\in\Omega$ and $0<j<N$ we introduce the jumps
$$
[u_{L,h}(x)]_j=u_{L,h}(x,y_j^+)-u_{L,h}(x,y_j^-).
$$
If $x$ is not in the interior of an element $K$ in the partition of $\Omega$ we choose
arbitrary an element containing it to evaluate $u_{L,h}$ (this is irrelevant because afterwards
we are going to integrate in $x$).

We will use the standard piecewise linear basis functions, namely, for $1\le j\le N-1$,
$$
\tau_j(y)= \left\{ \begin{array}{lcc}
\frac{y_{j+1}-y}{y_{j+1}-y_j} &   \mbox{if}  & y_j<y<y_{j+1} \\
\\ \frac{y-y_{j-1}}{y_j-y_{j-1}} &  \mbox{if} & y_{j-1} < y < y_j,
\end{array}
\right.
$$
$$
\tau_0(y)=\frac{y_1-y}{y_1} \qquad \mbox{if} \qquad 0<y<y_1,
$$
and
$$
\tau_N=\frac{y-y_{N-1}}{y_N-y_{N-1}} \qquad \mbox{if} \qquad  y_{n-1}<y<y_N.
$$
\begin{lemma}
For any $x\in\Omega$ we have
\begin{equation}
\label{Taylor approximation}
u_{L,h}(x,0)+\int_0^L\tau_0(y) y^{-\alpha}\sigma_{L,h,n+1}(x,y)dy
=\int_0^L y^{-\alpha}\sigma_{L,h,n+1}(x,y)dy.
\end{equation}
\end{lemma}

\begin{proof}
Since $u_{L,h}$ is piecewise constant one can see that
\begin{equation}
\label{barrow discreto}
u_{L,h}(x,L)=\sum_{j=1}^{N-1}[u_{L,h}(x)]_j+u_{L,h}(x,0).
\end{equation}
Let $K$ be the element containing $x$. For $1\le j\le N-1$, taking the function
$(\b0,\tau_j)$ supported in $K\times[y_{j-1},y_{j+1}]$ as test function in (\ref{mixedfem}), we have
$$
\int_0^L\int_K y^{-\alpha}\sigma_{L,h}\cdot(\b0,\tau_j) \,dx\,dy
- \int_0^L\int_K u_{L,h} div(\b0,\tau_j) \,dx\,dy=0
$$
and, since $\sigma_{L,h,n+1}(x,y)$ is independent of $x$ for $x\in K$, we obtain
$$
[u_{L,h}(x)]_j+ \int_0^Ly^{-\alpha}\sigma_{L,h,n+1}(x,y) \tau_j(y)\,dy=0.
$$
Analogously, using now $(0,\tau_N)$ yields
$$
u_{L,h}(x,L)=\int_0^L
y^{-\alpha}\sigma_{L,h,n+1}(x,y) \tau_N(y)\,dy.
$$
Therefore, replacing in \eqref{barrow discreto} we have
$$
\sum_{j=1}^N\int_0^Ly^{-\alpha}\sigma_{L,h,n+1}(x,y) \tau_j(y)\,dy=u_{L,h}(x,0)
$$
which immediately gives \eqref{Taylor approximation} because $\sum_{j=0}^N\tau_j\equiv 1$.
\end{proof}
To approximate the solution of \eqref{frac laplacian} given by $v(x)=u(x,0)$
we introduce
$$
v_{L,h}(x)=u_{L,h}(x,0)+\int_0^L\tau_0(y) y^{-\alpha}\sigma_{L,h,n+1}(x,y)dy.
$$
We also define $v_L(x)=u_L(x,0)$.

\begin{lemma}
\label{error vL-vLh}
$$
\|v_L-v_{L,h}\|_{L^2(\Omega)}
\le \frac1{\sqrt{1-\alpha}} L^{\frac{1-\alpha}2}
\|\bs-\bs_{L,h}\|_{L^2_{y^{-\alpha}}(\mathcal{C}_L)}.
$$
\end{lemma}
\begin{proof}
Since $u_L(x,L)=0$ and, recalling that
$\frac{\partial u_L}{\partial y}=-y^{-\alpha}\bs_{n+1}$,
we have
$$
v_L(x)=\int_0^Ly^{-\alpha}\bs_{n+1}(x,y)dy.
$$
Therefore, using \eqref{Taylor approximation} and the definition of $v_{L,h}$,
we obtain
$$
v_L(x)-v_{L,h}(x)
=\int_0^Ly^{-\alpha}(\bs_{n+1}(x,y)-\bs_{L,h,n+1}(x,y))dy,
$$
and, applying the Schwarz inequality,
$$
|v_L(x)-v_{L,h}(x)|^2
\le\left(\int_0^Ly^{-\alpha}dy\right)
\int_0^Ly^{-\alpha}|(\bs_{n+1}(x,y)-\bs_{L,h,n+1}(x,y))|^2dy,
$$
and integrating now in $x$ we conclude the proof.
\end{proof}
We can now prove the error estimate for the approximation of the solution of
the Fractional Laplacian.

\begin{theorem}
Under the hypotheses of Theorem \ref{teorema 5.3} we have
$$
\|v-v_{L,h}\|_{L^2(\Omega)}
\le Ch |\log h|^{\frac{3-\alpha}2} \|f\|_{\mathbb{H}^{1-s}(\Omega)},
$$
where the constant is as in Theorem \ref{teorema 5.3} an depends
also on $\alpha$.
\end{theorem}
\begin{proof}
From Lemma \ref{error vL-vLh} and, recalling that $L=C_1|\log h|$,
we have
$$
\|v_L-v_{L,h}\|_{L^2(\Omega)}
\le C |\log h|^{\frac{1-\alpha}2}
\|\bs-\bs_{L,h}\|_{L^2_{y^{-\alpha}}(\mathcal{C}_L)}
$$
where the constant depends on $\alpha$. Combining this estimate with
\eqref{error sigma definitivo} we obtain
\begin{equation}
\label{error vL-vLh 2}
\|v_L-v_{L,h}\|_{L^2(\Omega)}
\le C h|\log h|^{\frac{3-\alpha}2}
\|f\|_{\mathbb{H}^{1-s}(\Omega)}.
\end{equation}
It remains to estimate $v-v_L$. But, from the trace theorem given
in \cite[Proposition 2.5]{NOS} combined with \eqref{exponencial 2}
$$
\|v-v_L\|_{L^2(\Omega)}
\le C\|u-u_L\|_{H^1_{y^\alpha}(\mathcal{C})}
\le C e^{-\sqrt{\lambda_1}L/4}\|f\|_{\mathbb{H}^{-s}(\Omega)}
$$
and, from the definition of $L$ and $C_1$, we obtain
$$
\|v-v_L\|_{L^2(\Omega)}
\le C h\|f\|_{\mathbb{H}^{-s}(\Omega)}
$$
which combined with \eqref{error vL-vLh 2} concludes the proof.
\end{proof}
The next graphics show the order of the error $\|v-v_{L,h}\|_{L^2(\Omega)}$
for Problem \eqref{frac laplacian} with
$$
f(x_1,x_2)=(2\pi^2)^s\sin(\pi x_1)\sin(\pi x_2),
$$
which has as exact solution
$$
v(x_1,x_2)=\sin(\pi x_1)\sin(\pi x_2).
$$
\begin{figure}[H]
  \centering
    \includegraphics[width=7cm]{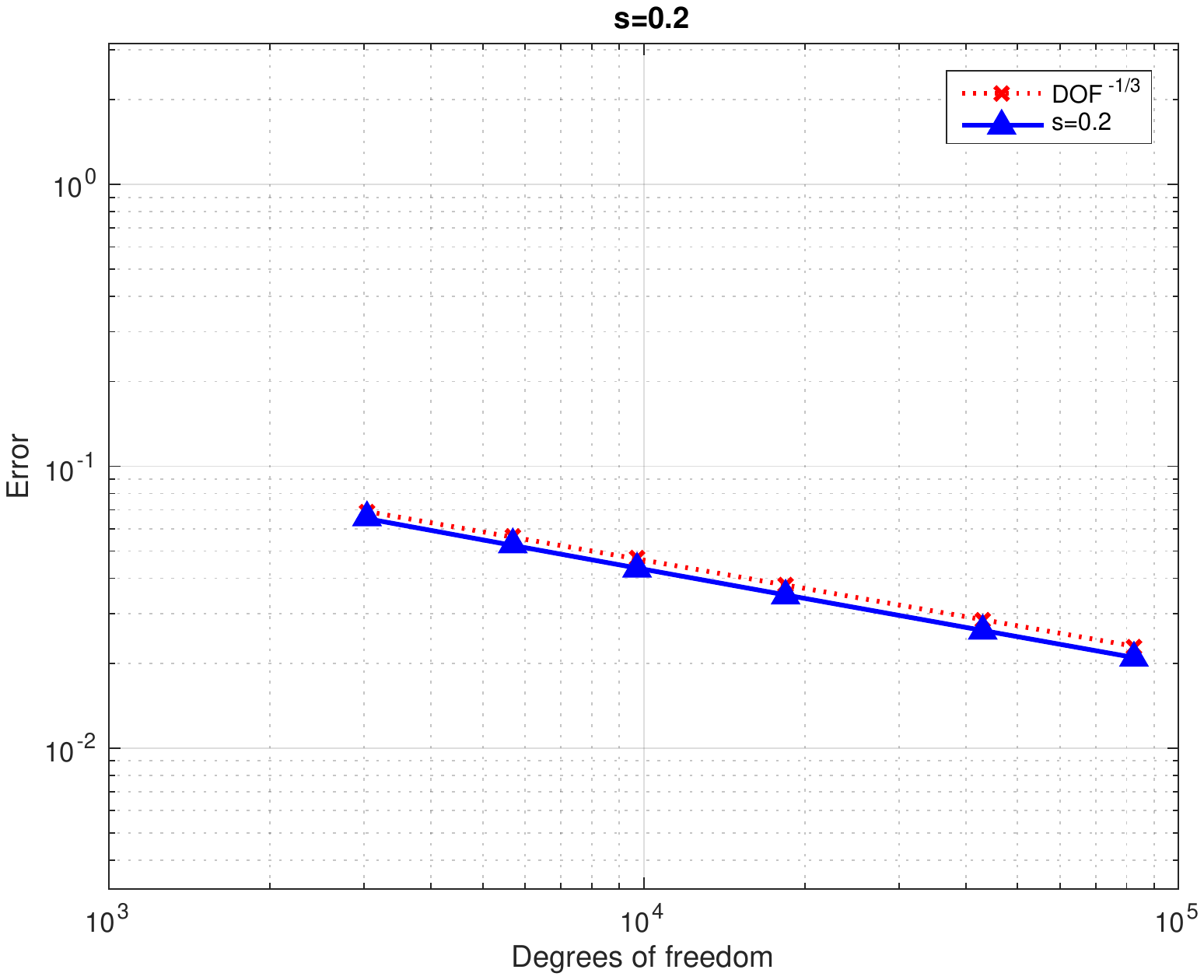}  \quad
    \includegraphics[width=7cm]{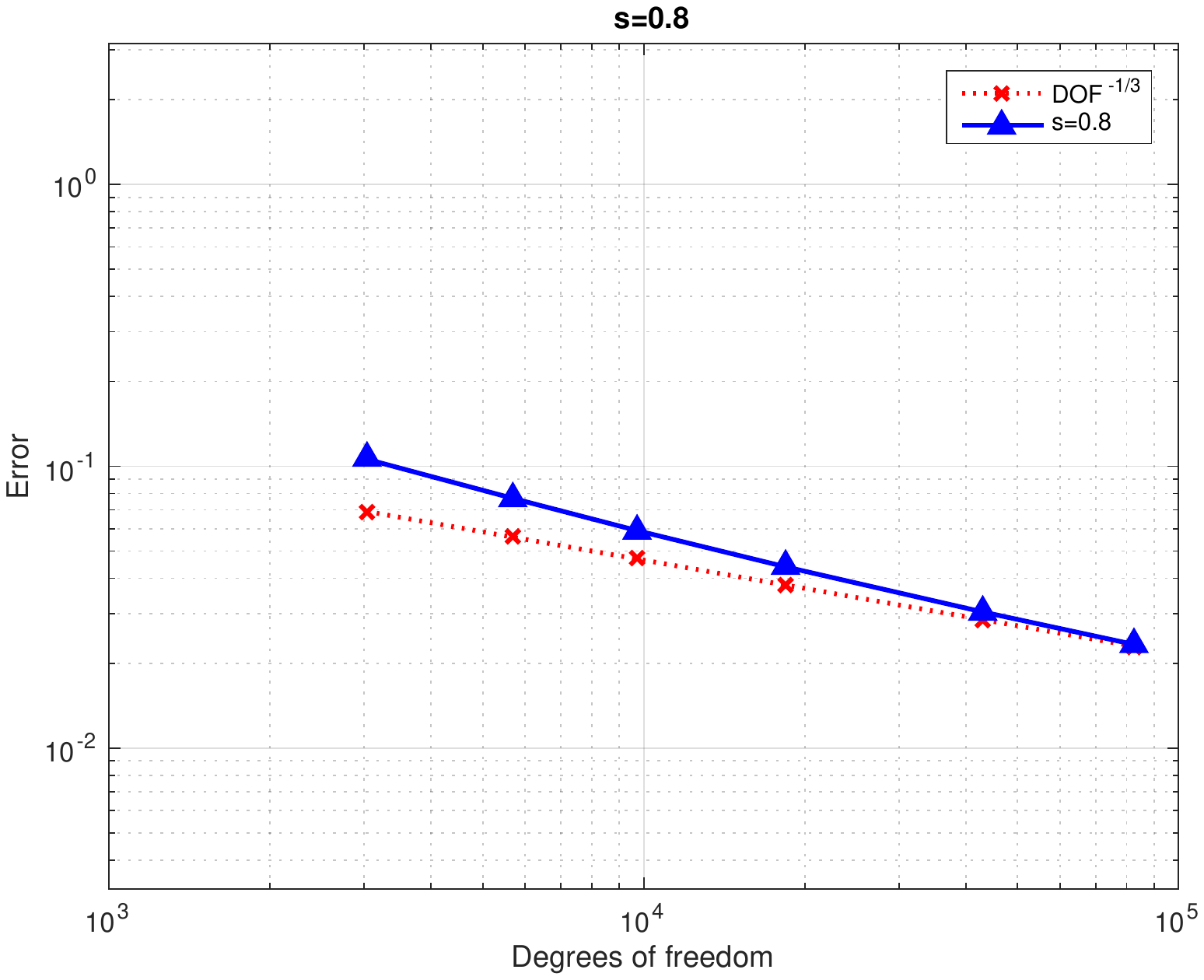}
  \caption{Rate of convergence: left $s=0.2$, right $s=0.8$.}
\end{figure}

\begin{remark}
The order of the error for the approximation
of $v$ in the $L^2$-norm is probably not the optimal possible.
Indeed, with a more complicated postprocessing one could approximate
the solution $u$ of Problem \eqref{CS1} with order almost $O(h)$ in
$H^1_{y^\alpha}(\mathcal{C})$ and, by the trace theorem
$\|v\|_{\mathbb{H}^s(\Omega)}\le C \|v\|_{H^1_{y^\alpha}(\mathcal{C})}$
proved in \cite[Proposition 2.5]{NOS}, one would have the same order
for the approximation of $v$ in the $\mathbb{H}^s$-norm. Therefore, it is reasonable
to expect a higher order in $L^2$. Let us mention also that, as far as we know,
such a higher order error estimate has not been proved either for the standard method
analyzed in \cite{NOS}.  This problem requires a different analysis and
will be the object of our further research.
\end{remark}

\bigskip

{\bf Acknowledgement}: We thank Enrique Ot\'arola for helpful comments.

\bibliographystyle{ams}

\end{document}